\documentclass[12pt,a4paper]{amsart}
\usepackage{amssymb}
\usepackage[all,cmtip]{xy}
\usepackage{hyperref}

\calclayout

\newcommand{\+}{\protect\nobreakdash-}
\renewcommand{\:}{\colon}

\newcommand{\rarrow}{\longrightarrow}

\newcommand{\ot}{\otimes}

\newcommand{\bu}{{\text{\smaller\smaller$\scriptstyle\bullet$}}}

\DeclareMathOperator{\Hom}{Hom}
\DeclareMathOperator{\im}{im}

\newcommand{\Modl}{{\operatorname{\mathsf{--Mod}}}}
\newcommand{\Comodl}{{\operatorname{\mathsf{--Comod}}}}
\newcommand{\Modr}{{\operatorname{\mathsf{Mod--}}}}
\newcommand{\Modrfl}{{\operatorname{\mathsf{Mod_{fl}--}}}}
\newcommand{\Bimod}{{\operatorname{\mathsf{--Bimod--}}}}
\newcommand{\Bimodrfl}{{\operatorname{\mathsf{--Bimod_{fl}--}}}}
\newcommand{\Bimodlrfl}{{\operatorname{\mathsf{--{}_{fl}Bimod_{fl}--}}}}

\newcommand{\Comon}{{\operatorname{\mathsf{--Comon}}}}
\newcommand{\Coalg}{{\operatorname{\mathsf{--Coalg}}}}
\newcommand{\Cocom}{{\operatorname{\mathsf{--Cocom}}}}
\newcommand{\Corings}{{\operatorname{\mathsf{--Corings}}}}

\newcommand{\Sets}{\mathsf{Sets}}

\newcommand{\fl}{\mathsf{fl}}
\newcommand{\bifl}{\mathsf{bifl}}
\newcommand{\obifl}{{\overline{\bifl}}}
\newcommand{\rfl}{\mathsf{rfl}}
\newcommand{\lrfl}{\mathsf{lrfl}}
\newcommand{\sur}{\mathsf{sur}}
\newcommand{\epi}{\mathsf{epi}}

\newcommand{\rop}{\mathrm{op}}
\newcommand{\id}{\mathrm{id}}

\newcommand{\sC}{\mathsf C}
\newcommand{\sD}{\mathsf D}
\newcommand{\sE}{\mathsf E}
\newcommand{\sK}{\mathsf K}
\newcommand{\sL}{\mathsf L}
\newcommand{\sM}{\mathsf M}
\newcommand{\sS}{\mathsf S}

\newcommand{\boZ}{\mathbb Z}

\newcommand{\Section}[1]{\bigskip\section{#1}\medskip}
\theoremstyle{plain}
\newtheorem{thm}{Theorem}[section]

\newtheorem{lem}[thm]{Lemma}
\newtheorem{prop}[thm]{Proposition}

\theoremstyle{definition}

\newtheorem{rem}[thm]{Remark}

\begin{document}

\title{The categories of corings and coalgebras \\ over a ring are
locally countably presentable}

\author{Leonid Positselski}

\address{Institute of Mathematics, Czech Academy of Sciences \\
\v Zitn\'a~25, 115~67 Praha~1 \\ Czech Republic}

\email{positselski@math.cas.cz}

\begin{abstract}
 For any commutative ring $R$, we show that the categories of
$R$\+coal\-ge\-bras and cocommutative $R$\+coalgebras are locally
$\aleph_1$\+presentable, while the categories of $R$\+flat
$R$\+coalgebras are $\aleph_1$\+accessible.
 Similarly, for any associative ring $R$, the category of $R$\+corings
is locally $\aleph_1$\+presentable, while the category of
$R$\+$R$\+bimodule flat $R$\+corings is $\aleph_1$\+accessible.
 The cardinality of the ring $R$ can be arbitrarily large.
 We also discuss $R$\+corings with surjective counit and flat kernel.
 The proofs are straightforward applications of an abstract
category-theoretic principle going back to Ulmer.
 For right or two-sided $R$\+module flat $R$\+corings, our cardinality
estimate for the accessibility rank is not as good.
 A generalization to comonoid objects in accessible monoidal categories
is also considered.
\end{abstract}

\maketitle

\tableofcontents

\section*{Introduction}
\medskip

 The aim of this paper is to close a gaping hole in contemporary
literature concerning the local presentability rank of the categories
of coalgebras and corings over a fixed ring~$R$.
 There is little really new material in this paper, as our results
go back to the unpublished 1977 preprint of Ulmer~\cite{Ulm}.
 The results of this paper concerning various classes of \emph{flat}
coalgebras and corings are relatively more original.

 Some ideas of the preprint~\cite{Ulm} were taken up and developed in
a different form in the 1984 dissertation of Bird~\cite{Bir};
but that remained unpublished, too.
 In the end, it appears that Ulmer's ideas were almost completely
forgotten and not incorporated into the contemporary body of knowledge.
 This created the gaping hole mentioned in the previous paragraph.
 We hope that this paper will fare better than Ulmer's preprint and
people will remember the idea now.

 One approach used in contemporary literature goes back to
the 1974 paper of Barr~\cite[Theorem~3.1]{Bar}.
 This theorem claims that, for any commutative ring $R$, any
$R$\+coalgebra $C$ is the directed union of its subcoalgebras
$C'\subset C$ such that $C'$ is a pure $R$\+submodule in $C$ and
the cardinality of $C'$ does not exceed the cardinality of $R$
plus~$\aleph_0$.
 The necessity to require the subcoalgebra $C'$ to be a pure
$R$\+submodule in $C$ arises from the well-known difficulty with
nonexactness of the tensor product in the theory of coalgebras or
corings over a ring.
 Basically, the very notion of a subcoalgebra does not make much
sense without the purity condition.
 It was emphasized in the present author's recent paper~\cite{Pres}
that the purification-based approach to accessibility often leads to
suboptimal cardinality bounds.

 The references to~\cite{Bar} in connection with the claim that
the category of cocommutative $R$\+coalgebras is locally
$\kappa$\+presentable for any infinite cardinal $\kappa>|R|+\aleph_0$
can be found in~\cite[Theorem~2.2]{Rap} and~\cite[Section~3.4]{RaRi}.
 The question of finding the best possible cardinality estimate for
the local presentability rank of this category was posed
in~\cite[Remark~2.3]{Rap}.

 A category-theoretic approach to corings and comodules was
suggested by Porst in the paper~\cite{Por}.
 Unfortunately, that paper does not reflect the understanding of
the locally presentable/accessible category theory involved that
could be found in Ulmer~\cite{Ulm}.
 The result is that~\cite[Theorem~9]{Por} claims local presentability
of the categories of coalgebras and corings without any explicit
cardinality bound.

 The paper~\cite[Section~3]{Ste} treats the categories of $R$\+flat
cocommutative $R$\+coal\-ge\-bras for Pr\"ufer domains~$R$.
 In particular, \cite[Lemma~3.2]{Ste} claims that the category of
$R$\+flat $R$\+coalgebras is locally presentable, with the argument
based on~\cite{Por}.
 Once again, there is no explicit cardinality bound.

 The classical case of coalgebras over a field~$k$ is much easier
than the general case of a ground ring~$R$.
 It is well-known that any coassociative coalgebra over~$k$ is
the directed union of its finite-dimensional
subcoalgebras~\cite[Theorem~2.2.1]{Swe}.
 It follows that the categories of coassociative coalgebras (with
or without cocommutativity or counitality) are locally finitely
presentable, with finitely presentable objects being precisely all
the finite-dimensional coalgebras.

 The analogous assertion is \emph{not} true for Lie coalgebras
over a field~$k$ (e.~g., in characteristic~$0$, the coalgebra dual
to the topological Lie algebra $k[[z]]d/dz$ of vector fields on
the one-dimensional formal disk contains \emph{no} nonzero proper
subcoalgebras).
 Still, one can easily show that any Lie coalgebra over~$k$
is the $\aleph_1$\+directed union of its subcoalgebras of at most
countable dimension.
 Hence the category of Lie coalgebras over~$k$ is locally
$\aleph_1$\+presentable, and its $\aleph_1$\+presentable objects
are precisely all the at most countably-dimensional Lie coalgebras.

 Let us now describe the content of the present paper.
 Very generally, we consider a monoidal category $\sM$ and a pair of
infinite cardinals $\lambda<\kappa$ such that $\kappa$ is regular,
the underlying category of $\sM$ is $\kappa$\+accessible and has
colimits of $\lambda$\+indexed chains, and the tensor product functor
$\ot\:\sM\times\sM\rarrow\sM$ preserves $\kappa$\+directed colimits.
 Under these assumptions, we show that the category of (coassociative,
counital) comonoid objects in $\sM$ is $\kappa$\+accessible, and its
$\kappa$\+presentable objects are precisely all the comonoid structures
on $\kappa$\+presentable objects of~$\sM$.
 When $\sM$ is a symmetric monoidal category, the same assertions apply
to the category of cocommutative comonoid objects in~$\sM$.

 In particular, given a commutative ring $R$, we consider the category of
coassociative, counital $R$\+coalgebras $R\Coalg$ and
the category of cocommutative, coassociative, counital
$R$\+coalgebras $R\Cocom$.
 We show that both the categories are locally $\aleph_1$\+presentable.
 The $\aleph_1$\+presentable objects are the coalgebras that are
$\aleph_1$\+presentable \emph{as $R$\+modules}.
 Furthermore, we point out that similar results hold for noncounital
or noncoassociative coalgebras, for Lie coalgebras, for conilpotent
coalgebras, and for DG\+coalgebras.

 We also consider the full subcategories $R\Coalg_\fl\subset R\Coalg$
and $R\Cocom_\fl\subset R\Cocom$ consisting of the coalgebras that are
\emph{flat as $R$\+modules}.
 We show that both the categories $R\Coalg_\fl$ and $R\Cocom_\fl$ are
$\aleph_1$\+accessible.
 The $\aleph_1$\+presentable objects of these categories are
the coalgebras that are flat and $\aleph_1$\+presentable as
$R$\+modules.
 The similar assertions hold for the categories of noncounital or
noncoassociative coalgebras, Lie coalgebras, conilpotent coalgebras,
and DG\+coalgebras.

 Given a commutative ring~$k$ and an associative $k$\+algebra~$R$,
we consider the category $R_k\Corings$ of coassociative, counital
corings over $R$ (i.~e., comonoid objects in the monoidal category
of $R$\+$R$\+bimodules over~$k$).
 We show that the category $R_k\Corings$ is locally
$\aleph_1$\+presentable.
 The $\aleph_1$\+presentable objects are the corings that are
$\aleph_1$\+presentable \emph{as $R$\+$R$\+bimodules over~$k$}
(i.~e., as modules over the ring $R\ot_kR^\rop$).
 Furthermore, we consider the category $R_k\Corings_\bifl$ of
$R$\+corings that are \emph{flat as $R\ot_kR^\rop$\+modules}.
 We show that the category $R_k\Corings_\bifl$ is
$\aleph_1$\+accessible, and its $\aleph_1$\+presentable objects
are the corings that are flat and $\aleph_1$\+presentable as
$R\ot_kR^\rop$\+modules.

 As a variation on the previous result, we consider \emph{corings with
flat kernel}, i.~e., the $R$\+corings $C$ such that the counit map
$\epsilon\:C\rarrow R$ is surjective and its kernel $\overline C$ is
a flat $R$\+$R$\+bimodule over~$k$.
 Such corings appear in connection with the Burt--Butler theory of
bocses~\cite[\S3]{BB}, \cite[Definition~4.19 and Theorem~4.20]{Kuel}.
 Assuming that the ring $R\ot_kR^\rop$ is countably Noetherian (e.~g.,
just Noetherian in the usual sense), our result tells that the category
$R_k\Corings_\obifl$ of $R$\+corings with flat kernel is
$\aleph_1$\+accessible, and its $\aleph_1$\+presentable objects are
the corings with flat kernel that are $\aleph_1$\+presentable
(equivalently, $<\aleph_1$\+generated) as $R\ot_kR^\rop$\+modules.

 More natural flatness conditions on an $R$\+coring $C$ are that
$C$ be flat as a right $R$\+module, or both as a right $R$\+module
and as a left $R$\+module (but not necessarily as a bimodule).
 The former condition characterizes the $R$\+corings $C$ for which
the category of left $C$\+comodules is abelian with an exact forgetful
functor to the category of left $R$\+modules~\cite[Sections~18.6,
18.14, and~18.16]{BW}, \cite[Proposition~2.12(a)]{Prev},
\cite[Lemma~2.1]{Pflcc}; hence its obvious importance.
 In this context, our results are less impressive, in that
the cardinality estimate is not any better than the one obtainable
with the purification-based approach.
 The problem is to obtain a good bound for the accessibility rank
of the additive category $R\Bimodrfl R$ of right $R$\+flat
$R$\+$R$\+bimodules and the additive category $R\Bimodlrfl R$ of left
and right $R$\+flat $R$\+$R$\+bimodules.

 Still, we explain how our category-theoretic approach can be applied
to the questions from the previous paragraph, if only for illustrative
purposes.
 The conclusion is that the category $R_k\Corings_\rfl$ of
right $R$\+flat $R$\+corings over~$k$ and the category
$R_k\Corings_\lrfl$ of left and right $R$\+flat $R$\+corings over~$k$
are $\kappa$\+accessible for any regular cardinal $\kappa>|R|+\aleph_0$.
 The $\kappa$\+presentable objects are the corings (with the respective
flatness property) whose underlying set has cardinality less
than~$\kappa$.

 We do not discuss comodules in this paper.
 An extensive treatment of accessibility properties of comodule
categories can be found in the paper~\cite[Sections~2\+-6]{Pflcc};
see in particular~\cite[Theorem~3.1 and Remark~3.2]{Pflcc}.

 Finally, let us say a few words about our motivation.
 Why is it important to know that the category of $R$\+coalgebras is
locally $\aleph_1$\+presentable, rather than just locally
$\kappa$\+presentable for some big enough cardinal~$\kappa$\,?
 One answer is that, together with a good bound on the presentability
rank, we obtain a description of the related full subcategory of
$\kappa$\+presentable objects.
 Over a field~$k$, the classical theorem of Sweedler tells that every
coassociative coalgebra is the directed union of its finite-dimensional
subcoalgebras.
 Over an arbitrary commutative ring~$R$, our
Theorem~\ref{coalgebras-locally-presentable-theorem} tells that
every coalgebra is a directed (in fact, $\aleph_1$\+directed) colimit
of coalgebras whose underlying $R$\+modules are countably presentable.

 From the module-theoretic and coalgebra-theoretic perspective, this
becomes particularly important for $R$\+flat $R$\+coalgebras.
 According to our Theorem~\ref{flat-coalgebras-accessible-theorem}, every
flat coalgebra is a directed (actually, $\aleph_1$\+directed) colimit
of coalgebras whose underlying $R$\+modules are flat \emph{and}
countably presentable.
 Any countably presentable flat module has projective dimension at
most~$1$ \,\cite[Corollary~2.23]{GT}.
 Moreover, any $\aleph_m$\+presentable flat module has projective
dimension at most~$m$ \,\cite[Proposition~5.3]{Jen},
\cite[Corollary~2.4]{Pres}.
 Whenever one is interested in the functor $\Hom_R(C,{-})$ for
an $R$\+coalgebra~$C$ (e.~g., in connection with
the \emph{$C$\+contramodules}~\cite{Psemi,Prev,Pflcc}), the question
of the projective dimension of the $R$\+module $C$ becomes
singularly important.
 We refer to the paper~\cite[Corollary~6.6, Corollary~9.2, and
Theorem~10.2]{PS6} for an example of a context where it is helpful to 
know that all flat objects of a certain class are directed colimits of
flat objects of finite projective dimension.

\subsection*{Acknowledgement}
 I~am grateful to an anonymous referee for the suggestion to include
a discussion of the categories of comonoids in accessible monoidal
categories (in Section~\ref{monoidal-secn}).
 The author is supported by the GA\v CR project 23-05148S and
the Czech Academy of Sciences (RVO~67985840).

\Section{Category-Theoretic Preliminaries} \label{prelim-secn}

 We use the book~\cite{AR} as the background reference source on
locally presentable and accessible categories.
 We refer to~\cite[Definition~1.4, Theorem and Corollary~1.5,
Definition~1.13(1), and Remark~1.21]{AR} for a relevant discussion
of \emph{$\kappa$\+directed} vs.\ \emph{$\kappa$\+filtered colimits}
for a regular cardinal~$\kappa$.
 Let us only recall that a poset is said to be
\emph{$\kappa$\+directed} if every its subset of the cardinality
smaller than~$\kappa$ has an upper bound.

 Let $\kappa$~be a regular cardinal and $\sK$ be a category with
$\kappa$\+directed (equivalently, $\kappa$\+filtered) colimits.
 An object $S\in\sK$ is said to be
\emph{$\kappa$\+presentable}~\cite[Definitions~1.1 and~1.13(2)]{AR}
if the functor $\Hom_\sK(S,{-})\:\sK\rarrow\Sets$ preserves
$\kappa$\+directed colimits.
 We denote the full subcategory of $\kappa$\+presentable objects by
$\sK_{<\kappa}\subset\sK$.

 The category $\sK$ is called
\emph{$\kappa$\+accessible}~\cite[Definition~2.1]{AR} if there is
a \emph{set} of $\kappa$\+presentable objects $\sS\subset\sK$ such that
all the objects of $\sK$ are $\kappa$\+directed colimits of objects
from~$\sS$.
 If this is the case, then the $\kappa$\+presentable objects of $\sK$
are precisely all the retracts of the objects from~$\sS$.
 A $\kappa$\+accessible category where all colimits exist is called
\emph{locally $\kappa$\+presentable}~\cite[Definition~1.17 and
Theorem~1.20]{AR}.

 $\aleph_0$\+presentable objects are called \emph{finitely
presentable}, $\aleph_0$\+accessible categories are called
\emph{finitely accessible}~\cite[Remark~2.2(1)]{AR}, and locally
$\aleph_0$\+presentable categories are called \emph{locally finitely
presentable}~\cite[Definition~1.9 and Theorem~1.11]{AR}.
 We call $\aleph_1$\+presentable objects \emph{countably presentable},
$\aleph_1$\+accessible categories \emph{countably accessible}, and
locally $\aleph_1$\+presentable categories \emph{locally countably
presentable}.

\begin{prop} \label{product-proposition}
 Let $\kappa$~be a regular cardinal and $(\sK_\xi)_{\xi\in\Xi}$ be
a family of $\kappa$\+accessible categories, indexed by a set\/ $\Xi$
of the cardinality smaller than~$\kappa$.
 Then the Cartesian product\/ $\sK=\prod_{\xi\in\Xi}\sK_\xi$ is
also a $\kappa$\+accessible category.
 The $\kappa$\+presentable objects of\/ $\sK$ are precisely all
the collections of objects $(S_\xi\in\sK_\xi)_{\xi\in\Xi}$ such that
$S_\xi\in(\sK_\xi)_{<\kappa}$ for every $\xi\in\Xi$.
\end{prop}

\begin{proof}
 This is a corrected version of~\cite[proof of Proposition~2.67]{AR}.
 See~\cite[Proposition~2.1]{Pacc} for the details.
\end{proof}

 In the following three theorems, we consider a regular
cardinal~$\kappa$ and a smaller infinite cardinal $\lambda<\kappa$
(so~$\kappa$ has to be uncountable).
 A \emph{$\lambda$\+indexed chain} (of objects and morphisms) in
a category $\sK$ is a directed diagram
$(K_i\to K_j)_{0\le i<j<\lambda}$ indexed by the ordered
set~$\lambda$.
 In the applications in the subsequent sections of this paper, we will
be mostly interested in the case $\lambda=\aleph_0$ and
$\kappa=\aleph_1$.

 Let $\sK$ and $\sL$ be two categories, and let $F$, $G\:\sK
\rightrightarrows\sL$ be a pair of parallel functors.
 The \emph{inserter category}~\cite[Section~2.71]{AR} of the pair of
parallel functors $F$, $G$ is the category $\sD$ whose objects are
pairs $(K,\phi)$, where $K\in\sK$ is an object and $\phi\:F(K)
\rarrow G(K)$ is a morphism in~$\sL$.
 The morphisms in $\sD$ are defined in the obvious way.

\begin{thm} \label{inserter-theorem}
 Let $\kappa$~be a regular cardinal and $\lambda<\kappa$ be a smaller
infinite cardinal.
 Let\/ $\sK$ and\/ $\sL$ be $\kappa$\+accessible categories where
colimits of $\lambda$\+indexed chains exist.
 Let $F$, $G\:\sK\rarrow\sL$ be a pair of parallel functors preserving
$\kappa$\+directed colimits.
 Assume further that the functor $F$ preserves colimits of
$\lambda$\+indexed chains and takes $\kappa$\+presentable objects
to $\kappa$\+presentable objects.
 Then the inserter category\/ $\sD$ of the pair of functors $F$, $G$ is
$\kappa$\+accessible.
 The $\kappa$\+presentable objects of\/ $\sD$ are precisely all
the pairs $(S,\phi)\in\sD$ with $S\in\sK_{<\kappa}$.
\end{thm}

\begin{proof}
 This result goes back to~\cite[Theorem~3.8, Corollary~3.9, and
Remark~3.11(II)]{Ulm}.
 For a recent exposition, see~\cite[Theorem~4.1]{Pacc}.

 A warning is due that in the formulations of the theorems both
in~\cite{Ulm} and in~\cite{Pacc} it is also assumed that the functor
$G$ preserves colimits of $\lambda$\+indexed chains.
 This assumption is not actually used in the proof (cf.~\cite[final
paragraph of the proof of Lemma~4.5]{Pacc}.
 This weakening of the assumptions of the inserter theorem is rarely
useful, but we will use it in
Theorem~\ref{comonoids-accessible-theorem}.
\end{proof}

 Let $\sK$ and $\sL$ be two categories, $F$, $G\:\sK\rightrightarrows
\sL$ be a pair of parallel functors, and $\phi$, $\psi\:
F\rightrightarrows G$ be a pair of parallel natural transformations.
 The \emph{equifier category}~\cite[Lemma~2.76]{AR} of the pair of
parallel natural transformations $\phi$ and~$\psi$ is the full
subcategory $\sE\subset\sK$ consisting of all objects $E\in\sK$ for
which the morphisms $\phi_E$ and $\psi_E\:F(E)\rightrightarrows G(E)$
are equal to each other in $\sL$, that is $\phi_E=\psi_E$.

\begin{thm} \label{equifier-theorem}
 Let $\kappa$~be a regular cardinal and $\lambda<\kappa$ be a smaller
infinite cardinal.
 Let\/ $\sK$ and\/ $\sL$ be $\kappa$\+accessible categories where
colimits of $\lambda$\+indexed chains exist.
 Let $F$, $G\:\sK\rarrow\sL$ be a pair of parallel functors preserving
$\kappa$\+directed colimits.
 Assume further that the functor $F$ preserves colimits of
$\lambda$\+indexed chains and takes $\kappa$\+presentable objects
to $\kappa$\+presentable objects.
 Let $\phi$ and $\psi\:F\rarrow G$ be a pair of parallel natural
transformations.
 Then the equifier category\/ $\sE$ of the pair of natural
transformations $\phi$, $\psi$ is $\kappa$\+accessible.
 The $\kappa$\+presentable objects of\/ $\sE$ are precisely all
the objects of\/ $\sE$ that are $\kappa$\+presentable in\/~$\sK$.
\end{thm}

\begin{proof}
 This also goes back to~\cite[Theorem~3.8, Corollary~3.9, and
Remark~3.11(II)]{Ulm}.
 For a recent exposition, see~\cite[Theorem~3.1]{Pacc}.
 Once again, in the formulations of the theorems both in~\cite{Ulm}
and in~\cite{Pacc} it is also assumed that the functor $G$ preserves
colimits of $\lambda$\+indexed chains.
 This assumption is not actually used in the proof (cf.~\cite[final
paragraph of the proof of Proposition~3.2]{Pacc}).
 This weakening of the assumptions of the equifier theorem is rarely
useful, but we will use it in
Theorem~\ref{comonoids-accessible-theorem}.
\end{proof}

 Let $\sK_1$, $\sK_2$, and $\sL$ be three categories, and
$F_1\:\sK_1\rarrow\sL$ and $F_2\:\sK_2\rarrow\sL$ be two functors.
 The \emph{pseudopullback category} of the pair of functors
$F_1$, $F_2$ is the category $\sC$ whose objects are triples
$(K_1,K_2,\theta)$, where $K_1\in\sK_1$ and $K_2\in\sK_2$ are
objects, and $\theta\:F_1(K_1)\simeq F_2(K_2)$ is an isomorphism
in~$\sL$.
 The morphisms in $\sC$ are defined in the obvious way.

\begin{thm} \label{pseudopullback-theorem}
 Let $\kappa$~be a regular cardinal and $\lambda<\kappa$ be a smaller
infinite cardinal.
 Let\/ $\sK_1$, $\sK_2$, and\/ $\sL$ be $\kappa$\+accessible categories
where colimits of $\lambda$\+indexed chains exist.
 Let $F_1\:\sK_1\rarrow\sL$ and $F_2\:\sK_2\rarrow\sL$ be two functors
preserving $\kappa$\+directed colimits and colimits of
$\lambda$\+indexed chains.
 Assume further that the functors $F_1$ and $F_2$ take
$\kappa$\+presentable objects to $\kappa$\+presentable objects.
 Then the pseudopullback\/ $\sC$ of the pair of functors $F_1$, $F_2$
is $\kappa$\+accessible.
 The $\kappa$\+presentable objects of\/ $\sC$ are precisely all
the triples $(S_1,S_2,\theta)\in\sC$ with $S_1\in(\sK_1)_{<\kappa}$
and $S_2\in(\sK_2)_{<\kappa}$.
\end{thm}

\begin{proof}
 This result, going back to~\cite[Remark~3.2(I), Theorem~3.8,
Corollary~3.9, and Remark~3.11(II)]{Ulm}, can be found
in~\cite[Pseudopullback Theorem~2.2]{RaRo} with the proof
in~\cite[Proposition~3.1]{CR}.
 For another exposition, see~\cite[Corollary~5.1]{Pacc}.
\end{proof}

 Before this section is finished, let us collect a couple of
well-known module-theoretic results concerning (local) presentability
and accessibility.
 Given an associative ring $R$, we denote by $R\Modl$ the abelian
category of left $R$\+modules and by $R\Modl_\fl\subset R\Modl$
the full subcategory of flat left $R$\+modules.

\begin{prop} \label{modules-locally-presentable}
 Let $R$ be an associative ring and $\kappa$~be a regular cardinal.
 Then the category $R\Modl$ of left $R$\+modules is locally
$\kappa$\+presentable.
 The $\kappa$\+presentable objects of $R\Modl$ are precisely all
the $R$\+modules that can be presented as the cokernel of a morphism
of free $R$\+modules with less than~$\kappa$ generators. \qed
\end{prop}

\begin{prop} \label{flat-modules-accessible}
 Let $R$ be an associative ring and $\kappa$~be a regular cardinal.
 Then the category $R\Modl_\fl$ of flat left $R$\+modules is
$\kappa$\+accessible.
 The $\kappa$\+presentable objects of $R\Modl_\fl$ are precisely all
the flat $R$\+modules that are $\kappa$\+presentable in $R\Modl$.
\end{prop}

\begin{proof}
 See~\cite[Proposition~10.2]{Pacc} or~\cite[Lemma~1.2]{Pflcc} for
some discussion.
\end{proof}

\Section{Comonoid Objects in Monoidal Categories} \label{monoidal-secn}

 We suggest the book~\cite[Section~VII.1]{McL} as a reference on
monoidal categories, and~\cite[Sections~VII.7 and~XI.1]{McL} for
symmetric monoidal categories.

 Let us just fix the notation.
 A \emph{monoidal category} $\sM$ is supposed to be associative and 
unital; the monoidal operation is denoted by $\ot\:\sM\times\sM\rarrow
\sM$ and called the \emph{tensor product}; the unit object is denoted
by $\mathbf1\in\sM$.
 The \emph{associativity and unitality constraints} are functorial
isomorphisms $(K\ot L)\ot M\simeq K\ot(L\ot M)$ and $\mathbf1\ot M\simeq
M\simeq M\ot\mathbf1$ given for all objects $K$, $L$, $M\in\sM$.
 A \emph{symmetric monoidal category} is also endowed with
a \emph{commutativity constraint}, which is a functorial isomorphism
$K\ot L\simeq L\ot K$ for all objects $K$, $L\in\sM$.

 We do not go into a discussion of the commutativity of pentagonal
and hexagonal coherence diagrams, as the questions of coherence play
no role in our exposition below.
 Instead, we will adopt a policy of benign neglect and write simply
$K\ot L\ot M$ as a common notation for either $(K\ot L)\ot M$ or
$K\ot(L\ot M)$.

 A \emph{comonoid object} in a monoidal category $\sM$ is an object
$C\in\sM$ endowed with two morphisms of \emph{comultiplication}
$\mu\:C\rarrow C\ot C$ and \emph{counit} $\epsilon\:C\rarrow\mathbf1$
satisfying the conventional coassociativity and counitality axioms.
 Specifically, the two compositions
$$
 C\rarrow C\ot C\rightrightarrows C\ot C\ot C
$$
of the comultiplication morphism with two morphisms induced by
the comultiplication morphism must be equal to each other,
$(\mu\ot\id_C)\circ\mu=(\id_C\ot\mu)\circ\mu$, and the two compositions
$$
 C\rarrow C\ot C\rightrightarrows C
$$
of the comultiplication morphism with the two morphisms induced by
the counit morphism must be equal to the identity morphism,
$(\epsilon\ot\id_C)\circ\mu=\id_C=(\id_C\ot\epsilon)\circ\mu$.

 In a symmetric monoidal category $\sM$, one can speak of cocommutative
comonoid objects.
 A comonoid object $C$ is said to be \emph{cocommutative} if
the composition
$$
 C\rarrow C\ot C\rarrow C\ot C
$$
of the comultiplication morphism with the commutativity constraint
morphism $\sigma_C\:C\ot C\rarrow C\ot C$ is equal to
the comultiplication morphism, $\sigma_C\circ\mu=\mu$.

 Morphisms of comonoid objects $C\rarrow D$ are defined in the obvious
way.
 We denote the category of comonoid objects in a monoidal category
$\sM$ by $\sM\Comon$.
 When $\sM$ is a symmetric monoidal category, the full subcategory of
cocommutative comonoid objects is denoted by $\sM\Cocom\subset
\sM\Comon$.

\begin{thm} \label{comonoids-accessible-theorem}
\textup{(a)} Let $\kappa$~be a regular cardinal, $\lambda<\kappa$ be
a smaller infinite cardinal, and\/ $\sM$ be a monoidal category.
 Assume that the underlying category of\/ $\sM$ is $\kappa$\+accessible,
that colimits of $\lambda$\+indexed chains exist in it, and that
the monoidal operation functor\/ $\ot\:\sM\times\sM\rarrow\sM$ preserves
$\kappa$\+directed colimits (in both arguments).
 Then the category of comonoid objects\/ $\sM\Comon$ is
$\kappa$\+accessible.
 An object of\/ $\sM\Comon$ is $\kappa$\+presentable if and only if
its underlying object is $\kappa$\+presentable in\/~$\sM$. \par
\textup{(b)} In the context of part~(a), assume additionally that\/
$\sM$ is a symmetric monoidal category.
 Then the category of cocommutative comonoid objects\/ $\sM\Cocom$ is
$\kappa$\+accessible.
 An object of\/ $\sM\Cocom$ is $\kappa$\+presentable if and only if
its underlying object is $\kappa$\+presentable in\/~$\sM$.
\end{thm}

\begin{proof}
 The argument is based on Theorems~\ref{inserter-theorem}
and~\ref{equifier-theorem} together with
Proposition~\ref{product-proposition}.
 Let us spell out the proof of part~(b); part~(a) is similar.

 In order to apply Theorem~\ref{inserter-theorem}, put $\sK=\sM$
and $\sL=\sM\times\sM$.
 Let $F\:\sK\rarrow\sL$ be the functor taking an object $C\in\sM$
to the pair of objects $(C,C)\in\sM\times\sM$, and let $G\:\sK
\rarrow\sL$ be the functor taking the object $C$ to the pair of
objects $(C\ot C,\>\mathbf1)\in\sM\times\sM$.
 Then the inserter category $\sD$ is the category of objects
$C\in\sM$ endowed with two morphisms $\mu\:C\rarrow C\ot C$ and
$\epsilon\:C\rarrow\mathbf1$.

 Let us spell out a couple of additional words about how this works.
 An object $C\in\sM$ as such is endowed with \emph{no} multiplication
or unit morphism structure.
 It is just an object of a monoidal category~$\sM$ (such as, e.~g.,
an $R$\+module, if $\sM$ is the monoidal category of modules over
a commutative ring~$R$).
 By the definition, an object of the inserter category $\sD$ is
an object $C\in\sK=\sM$ together with an arbitrary morphism
$F(C)\rarrow G(C)$ in $\sL=\sM\times\sM$.
 This means a morphism $(C,C)\rarrow(C\ot C,\>\mathbf1)$ in
$\sM\times\sM$, i.~e., an arbitrary pair of morphisms
$\mu\:C\rarrow C\ot C$ and $\epsilon\:C\rarrow\mathbf1$.
 So an object of $\sD$ is an object $C\in\sM$ endowed with two
pieces of additional data, viz., two arbitrary morphisms
$\mu\:C\rarrow C\ot C$ and $\epsilon\:C\rarrow\mathbf1$.
 The morphisms in $\sD$ are defined in the obvious way as
the morphisms in $\sM$ compatible with the additional structure
provided by the maps~$\mu$ and~$\epsilon$.

 By assumption and by Proposition~\ref{product-proposition}, both
the categories $\sK$ and $\sL$ are $\kappa$\+accessible.
 The assumptions of Theorem~\ref{inserter-theorem} are satisfied,
and we can conclude that the category $\sD$ is $\kappa$\+accessible.
 It is clear that colimits of $\lambda$\+indexed chains exist in
$\sD$ and are preserved by the forgetful functor $\sD\rarrow\sM$.
 Furthermore, Theorem~\ref{inserter-theorem} provides a description
of the full subcategory of $\kappa$\+presentable objects in~$\sD$.

 In order to apply Theorem~\ref{equifier-theorem}, put $\sK=\sD$
and $\sL=\sM^4=\sM\times\sM\times\sM\times\sM$.
 Let $F\:\sK\rarrow\sL$ be the functor taking a triple
$(C,\mu,\epsilon)\in\sD$ to the quadruple of objects
$(C,C,C,C)\in\sL$, and let $G\:\sK\rarrow\sL$ be the functor taking
the same object $(C,\mu,\epsilon)\in\sD$ to the quadruple of objects
$(C\ot C\ot C,\>C,\>C,\>C\ot C)\in\sL$.

 Let $\phi\:F\rarrow G$ be the natural transformation acting by
the following quadruple of morphisms $\phi_1\:C\rarrow C\ot C\ot C$,
\ $\phi_2\:C\rarrow C$, \ $\phi_3\:C\rarrow C$, and
$\phi_4\:C\rarrow C\ot C$.
 The morphism~$\phi_1$ is the composition $(\mu\ot\id_C)\circ\mu\:
C\rarrow C\ot C\ot C$.
 The morphism~$\phi_2$ is the composition $(\epsilon\ot\id_C)\circ\mu\:
C\rarrow C$.
 The morphism~$\phi_3$ is the composition $(\id_C\ot\epsilon)\circ\mu\:
C\rarrow C$.
 The morphism $\phi_4$ is the composition $\sigma_C\circ\mu\:C\rarrow
C\ot C$.

 Let $\psi\:F\rarrow G$ be the natural transformation acting by
the following quadruple of morphisms $\psi_1\:C\rarrow C\ot C\ot C$,
\ $\psi_2\:C\rarrow C$, \ $\psi_3\:C\rarrow C$, and
$\psi_4\:C\rarrow C\ot C$.
 The morphism~$\psi_1$ is the composition $(\id_C\ot\mu)\circ\mu\:
C\rarrow C\ot C\ot C$.
 The morphisms $\psi_2$ and~$\psi_3$ are the identity morphisms
$\id_C\:C\rarrow C$.
 The morphism~$\psi_4$ is the morphism $\mu\:C\rarrow C\ot C$.
 Then the equifier category $\sE$ is the category of cocommutative
comonoid objects in $\sM$, that is $\sE=\sM\Cocom$.

 The construction of the equifier category imposes the axioms of
coassociativty, counitality, and cocommutativity on the morphisms
$\mu\:C\rarrow C\ot C$ and $\epsilon\:C\rarrow\nobreak\mathbf1$.
 This is accomplished by passing to the full subcategory $\sE$ of
the category $\sK=\sD$.
 The full subcategory $\sE\subset\sD$ consists of all the objects
$(C,\mu,\epsilon)\in\sD$ for which the morphisms $\mu$ and~$\epsilon$
satisfy the coassociativity, counitality, and cocommutativity equations.

 By Proposition~\ref{product-proposition} and the discussion above,
both the categories $\sK=\sD$ and $\sL$ are $\kappa$\+accessible.
 The assumptions of Theorem~\ref{equifier-theorem} are satisfied,
and we can conclude that the category $\sE=\sM\Cocom$ is
$\kappa$\+accessible.
 Theorem~\ref{equifier-theorem} also provides the desired description
of the full subcategory of $\kappa$\+presentable objects in~$\sE$.
\end{proof}

\begin{rem} \label{other-comonoids-accessible-remark}
 Some variations on the theme of
Theorem~\ref{comonoids-accessible-theorem} are possible, producing
other examples of $\kappa$\+accessible categories of comonoid objects,
together with explicit descriptions of their full subcategories of
$\kappa$\+presentable objects.
 In particular, this applies to the categories of noncounital and/or
noncoassociative comonoids in~$\sM$.
 All one needs to do in these cases is to drop the related elements
from the proof of Theorem~\ref{comonoids-accessible-theorem} above.
\end{rem}

\Section{Coalgebras over a Commutative Ring} \label{coalgebras-secn}

 Let $R$ be a commutative ring.
 A (\emph{coassociative, counital}) \emph{coalgebra} over $R$ is
a comonoid object in the monoidal category of $R$\+modules $R\Modl$
with respect to the operation of tensor product over~$R$.
 In other words, a coalgebra $C$ is an $R$\+module endowed with two
maps of \emph{comultiplication} $\mu\:C\rarrow C\ot_RC$ and
\emph{counit} $\epsilon\:C\rarrow R$, which must be $R$\+linear maps
satisfying the conventional coassociativity and counitality axioms.
 Specifically, the two compositions
$$
 C\rarrow C\ot_R C\rightrightarrows C\ot_RC\ot_RC
$$
of the comultiplication map with two maps induced by
the comultiplication map must be equal to each other, that is
$(\mu\ot\id_C)\circ\mu=(\id_C\ot\mu)\circ\mu$, and the two compositions
$$
 C\rarrow C\ot_RC\rightrightarrows C
$$
of the comultiplication map with two maps induced by
the counit map must be equal to the identity map, that is
$(\epsilon\ot\id_C)\circ\mu=\id_C=(\id_C\ot\epsilon)\circ\mu$.

 An $R$\+coalgebra $C$ is said to be \emph{cocommutative} if
$\sigma_C\circ\mu=\mu$,
$$
 C\rarrow C\ot_RC\rarrow C\ot_RC,
$$
where $\sigma_C\:C\ot_RC\rarrow C\ot_RC$ is the map permuting
the tensor factors.
 The notion of a cocommutative $R$\+coalgebra makes sense because
$R\Modl$ is naturally a symmetric monoidal category.

 Morphisms of coalgebras $C\rarrow D$ are defined in the obvious way
as $R$\+linear maps compatible with the comultiplication and counit.
 We denote the category of $R$\+coalgebras by $R\Coalg$ and
the full subcategory of cocommutative $R$\+coalgebras by
$R\Cocom\subset R\Coalg$.

\begin{lem} \label{coalgebras-cocomplete}
 All colimits exist in the categories $R\Coalg$ and $R\Cocom$.
 The forgetful functors $R\Coalg\rarrow R\Modl$ and
$R\Cocom\rarrow R\Modl$ preserve colimits.
\end{lem}

\begin{proof}
 Let us discuss the colimits in $R\Coalg$; the situation in
$R\Cocom$ is similar.
 It suffices to check that coproducts and coequalizers exist in
$R\Coalg$ and are preserved by the forgetful functor to $R\Modl$.
 Given a family of $R$\+coalgebras $(C_\xi)_{\xi\in\Xi}$, it is easy
to construct a coalgebra structure on the direct sum
$\bigoplus_{\xi\in\Xi}C_\xi$ taken in $R\Modl$ and show that
the resulting coalgebra is the coproduct of $C_\xi$ in $R\Coalg$.
 Let us explain the coequalizers in more detail.

 Suppose given a pair of parallel morphisms of $R$\+coalgebras
$f$, $g\:C\rightrightarrows D$.
 We claim that the image $I$ of the $R$\+module map $f-g\:C\rarrow D$
is a coideal in~$D$.
 This means that the counit map $\epsilon_D\:D\rarrow R$
vanishes in restriction to~$I$ and the image of the composition
$I\rarrow D\rarrow D\ot_RD$ of the inclusion map $I\rarrow D$
with the comultiplication map $\mu_D\:D\rarrow D\ot_RD$ is contained
in the image of the map $I\ot_R D\oplus D\ot_R I\rarrow D\ot_RD$
induced by the inclusion map $I\rarrow D$.

 Indeed, one clearly has $\epsilon_D(f(c)-g(c))=\epsilon_D(f(c))-
\epsilon_D(g(c))=\epsilon_C(c)-\epsilon_C(c)=0$ for all $c\in C$.
 Now let the element $\mu_C(c)\in C\ot_RC$ be equal to
$\sum_{i=1}^n c'_i\ot c''_i$, where $c'_i$, $c''_i\in C$.
 Then we have $\mu_D(f(c)-g(c))=(f\ot f-g\ot g)(\mu_C(c))=
\sum_{i=1}^n (f(c'_i)\ot f(c''_i)-g(c'_i)\ot g(c''_i))=
\sum_{i=1}^n (f(c'_i)-g(c'_i))\ot f(c''_i)+
\sum_{i=1}^n g(c'_i)\ot (f(c''_i)-g(c''_i))\in\im(I\ot_RD\oplus
D\ot_RI)\subset D\ot_RD$, as desired.

 It follows that there is a unique $R$\+coalgebra structure on
the quotient $R$\+module $E=D/I$ such that the natural surjective map
$p\:D\rarrow E$ is a coalgebra morphism (see, e.~g.,
\cite[Section~2.4]{BW}, and notice that the image of the map
$I\ot_R D\oplus D\ot_R I\rarrow D\ot_RD$ is always contained in
the kernel of the map $p\ot p\:D\ot_RD\rarrow E\ot_RE$;
so any coideal in the sense explained above is also a coideal in
the sense of~\cite{BW}).
 One can readily check that the morphism $p\:D\rarrow E$ is
the coequalizer of the morphisms $f$ and~$g$ in $R\Coalg$.

 In fact, it is easy to see that the image of the map
$I\ot_R D\oplus D\ot_R I\rarrow D\ot_RD$ coincides with
the kernel of the map $p\ot p\:D\ot_RD\rarrow E\ot_RE$.
 So a coideal in the sense explained above is the same thing as
a coideal in the sense of~\cite{BW}.
\end{proof}

 The following theorem is the main result of this section.

\begin{thm} \label{coalgebras-locally-presentable-theorem}
 Let $R$ be a commutative ring. \par
\textup{(a)} The category of coassociative, counital $R$\+coalgebras
$R\Coalg$ is locally\/ $\aleph_1$\+pre\-sent\-able.
 The\/ $\aleph_1$\+presentable objects of $R\Coalg$ are precisely
all the $R$\+coalgebras whose underlying $R$\+modules are
countably presentable as objects of $R\Modl$. \par
\textup{(b)} The category of coassociative, cocommutative, counital
$R$\+coalgebras $R\Cocom$ is locally\/ $\aleph_1$\+presentable.
 The\/ $\aleph_1$\+presentable objects of $R\Cocom$ are precisely
all the cocommutative $R$\+coalgebras whose underlying $R$\+modules
are countably presentable as objects of $R\Modl$.
\end{thm}

\begin{proof}
 This is~\cite[Example~4.3]{Ulm}.
 Let us spell out a proof of part~(b) based on the techniques
from~\cite{Pacc} collected in Section~\ref{prelim-secn}.
 In view of Lemma~\ref{coalgebras-cocomplete}, it suffices to show that
the category $R\Cocom$ is $\aleph_1$\+accessible and describe its full
subcategory of $\aleph_1$\+presentable objects.

 Firstly we apply Theorem~\ref{inserter-theorem} (for $\kappa=\aleph_1$
and $\lambda=\aleph_0$).
 Put $\sK=R\Modl$ and $\sL=R\Modl\times R\Modl$.
 Let $F\:\sK\rarrow\sL$ be the functor taking an $R$\+module $C$ to
the pair of $R$\+modules $(C,C)\in R\Modl\times R\Modl$, and let
$G\:\sK\rarrow\sL$ be the functor taking the $R$\+module $C$ to
the pair of $R$\+modules $(C\ot_RC,\>R)\in R\Modl\times R\Modl$.
 Then the inserter category $\sD$ is the category of $R$\+modules $C$
endowed with two $R$\+linear maps $\mu\:C\rarrow C\ot_RC$ and
$\epsilon\:C\rarrow R$.

 By Propositions~\ref{modules-locally-presentable}
and~\ref{product-proposition}, both the categories $\sK$ and $\sL$
are $\aleph_1$\+accessible (in fact, locally $\aleph_1$\+presentable).
 The assumptions of Theorem~\ref{inserter-theorem} are satisfied,
and we can conclude that the category $\sD$ is $\aleph_1$\+accessible.
 Theorem~\ref{inserter-theorem} also provides a description of the full
subcategory of $\aleph_1$\+presentable objects in~$\sD$.

 Now we apply Theorem~\ref{equifier-theorem} (again for $\kappa=
\aleph_1$ and $\lambda=\aleph_0$).
 Put $\sK=\sD$ and $\sL=(R\Modl)^4=R\Modl\times R\Modl\times R\Modl
\times R\Modl$.
 Let $F\:\sK\rarrow\sL$ be the functor taking a triple $(C,\mu,\epsilon)
\in\sD$ to the quadruple of $R$\+modules $(C,C,C,C)\in\sL$,
and let $G\:\sK\rarrow\sL$ be the functor taking the same object
$(C,\mu,\epsilon)\in\sD$ to the quadruple of $R$\+modules
$(C\ot_RC\ot_RC,\>C,\>C,\>C\ot_RC)\in\sL$.

 Finally, we choose the pair of natural transformations
$\phi$, $\psi\:F\rarrow G$ as in the proof of
Theorem~\ref{comonoids-accessible-theorem}.
 Then the equifier category $\sE$ is the category of coassociative,
counital, cocommutative coalgebras $C$ over $R$, that is $\sE=R\Cocom$.

 By Propositions~\ref{modules-locally-presentable}
and~\ref{product-proposition}, and by the discussion above, both
the categories $\sK=\sD$ and $\sL$ are $\aleph_1$\+accessible.
 The assumptions of Theorem~\ref{equifier-theorem} are satisfied, and
we can conclude that the category $\sE=R\Cocom$ is
$\aleph_1$\+accessible.
 Theorem~\ref{equifier-theorem} also provides the desired description
of the full subcategory of $\aleph_1$\+presentable objects in~$\sE$.

 Essentially the same argument is applicable in the case of
the noncocommutative coalgebras (part~(a) of the theorem).
 One only needs to drop the elements related to cocommutativity in
the proof above (i.~e., the morphisms $\phi_4$ and~$\psi_4$,
and the related components of the category $\sL$ and the functors
$F$ and~$G$ in the context of Theorem~\ref{equifier-theorem}).
 So one takes $\sL=(R\Modl)^3=R\Modl\times R\Modl\times R\Modl$ when
applying Theorem~\ref{equifier-theorem} for the proof of part~(a).

 Alternatively, the assertions about $\aleph_1$\+accessibility of
the categories $R\Coalg$ and $R\Cocom$ together with the descriptions
of $\aleph_1$\+presentable objects in these categories can be obtained
as particular cases of Theorem~\ref{comonoids-accessible-theorem} for
the monoidal category $\sM=R\Modl$ with the monoidal operation functor
of tensor product $\ot=\ot_R$.
\end{proof}

\begin{rem} \label{other-coalgebras-locally-presentable-remark}
 Numerous variations on the theme of
Theorem~\ref{coalgebras-locally-presentable-theorem} are possible,
producing other examples of locally $\aleph_1$\+presentable
categories of $R$\+coalgebras.
 Let us mention some of them.

\smallskip
 (1)~Noncounital and/or noncoassociative $R$\+coalgebras also form
locally $\aleph_1$\+pre\-sent\-able categories.
 All one needs to do in these cases is to drop the related elements
from the proof of Theorem~\ref{coalgebras-locally-presentable-theorem}
above, or refer to Remark~\ref{other-comonoids-accessible-remark}.

\smallskip
 (2)~A Lie coalgebra $L$ over a commutative ring $R$ is
an $R$\+module endowed with an $R$\+linear map
$\delta\:L\rarrow\bigwedge^2_RL$ that can be extended to
a differential defining a DG\+algebra structure on
the exterior algebra $\bigwedge^*_RL$ of the $R$\+module~$L$.
 Explicitly, denoting the map~$\delta$ in a Sweedler-style
notation~\cite[Section~1.2]{Swe} by $l\longmapsto\delta(l)=
l_{\{1\}}\wedge l_{\{2\}}\in\bigwedge^2_RL$ for $l\in L$,
the Lie coalgebra axiom (the dual version of the Jacobi identity)
is the equation
$$
 l_{\{1\}\{1\}}\wedge l_{\{1\}\{2\}}\wedge l_{\{2\}}
 = l_{\{1\}}\wedge l_{\{2\}\{1\}}\wedge l_{\{2\}\{2\}}
$$
in~$\bigwedge^3_RL$ \,\cite[Section~D.2.1]{Psemi}.

 The category of Lie coalgebras over an arbitrary commutative
ring $R$ is locally $\aleph_1$\+presentable.
 In order to obtain a proof, one needs to (drop the unitality and)
replace the tensor power functors $C\longmapsto C\ot_RC$ and
$C\longmapsto C\ot_RC\ot_RC$ with the exterior power functors
$L\longmapsto\bigwedge^2_RL$ and $L\longmapsto\bigwedge^3_RL$ in
the proof of Theorem~\ref{coalgebras-locally-presentable-theorem}.

\smallskip
 (3)~A coassociative, noncounital coalgebra $D$ over a commutative
ring $R$ is said to be \emph{conilpotent} if for every element
$d\in D$ there exists an integer $n\ge1$ such that the element~$d$
is annihilated by the iterated comultiplication map
$\mu^{(n)}\:D\rarrow D^{\ot n+1}$, that is $\mu^{(n)}(d)=0$.
 In order to show that the category of conilpotent $R$\+coalgebras
is locally $\aleph_1$\+presentable, one can consider the following
inserter category~$\sD$.
 Put $\sK=\sL=R\Modl$, let $F\:\sK\rarrow\sL$ be the identity functor,
and let $G\:\sK\rarrow\sL$ be the functor assigning to every
$R$\+module $D\in R\Modl$ the $R$\+module $\bigoplus_{n=1}^\infty
D^{\ot n} = D\oplus D\ot_RD\oplus D\ot_RD\ot_RD\oplus\dotsb$.
 Then the suitable equifier full subcategory $\sE\subset\sD$ is
the category of conilpotent $R$\+coalgebras.

\smallskip
 (4)~A \emph{DG\+coalgebra} $C^\bu$ over $R$ is a comonoid object in
the (symmetric) monoidal category of \emph{complexes} of $R$\+modules
with respect to the operation of tensor product~$\ot_R$.
 In order to show that the category of DG\+coagebras over $R$ is
locally $\aleph_1$\+presentable, one needs to know that the category of
complexes of $R$\+modules is locally finitely presentable and its
$\aleph_1$\+presentable objects are precisely all the complexes of
$\aleph_1$\+presentable $R$\+modules~\cite[Lemma~1.5]{Pres}.
 Then one argues similarly to the proof of
Theorem~\ref{coalgebras-locally-presentable-theorem}, replacing
modules by complexes everywhere.
 Alternatively, one can refer to
Theorem~\ref{comonoids-accessible-theorem}.

\smallskip
 In each of the cases~(1\+-4), the description of the full subcategory
of $\aleph_1$\+presentable objects in the respective locally
$\aleph_1$\+presentable category of coalgebras is similar to the one in
Theorem~\ref{coalgebras-locally-presentable-theorem}.
\end{rem}

\begin{rem}
 Let us offer a sketchy and approximate explanation of the workings
of the proof of Theorem~\ref{coalgebras-locally-presentable-theorem},
based as it is on Theorems~\ref{inserter-theorem}
and~\ref{equifier-theorem}.
 Following Remark~\ref{other-coalgebras-locally-presentable-remark}(1),
we ignore the counit and consider the category of coassociative,
noncocommutative, noncounital coalgebras $C$ over
a fixed commutative ring~$R$.

\smallskip
 (1)~In order to illustrate the workings of the proof of
Theorem~\ref{inserter-theorem} (as spelled out in~\cite[proof of
Theorem~4.1]{Pacc}) in the situation at hand, let us start with
coalgebras that are not even coassociative.
 So we are interested in $R$\+modules $C$ endowed with an $R$\+module
map $\nu_C\:C\rarrow C\ot_RC$.
 How does one approach $C$ with countably presentable $R$\+modules $U$
endowed with $R$\+module maps $\nu_U\:U\rarrow U\ot_RU$\,?

 Suppose given an $R$\+module $C$ endowed with an $R$\+linear map
$\nu_C\:C\rarrow C\ot_RC$, a countably presentable $R$\+module $S$,
and an $R$\+module map $S\rarrow C$.
 Let us explain, roughly following~\cite[proof of Lemma~4.5]{Pacc},
how to construct a countably presentable $R$\+module $U$ with
an $R$\+module map $\nu_U\:U\rarrow U\ot_RU$, an $R$\+module map
$U\rarrow C$ compatible with $\nu_U$ and~$\nu_C$, and an $R$\+module
map $S\rarrow U$ making the triangular diagram $S\rarrow U\rarrow C$
commutative.

 Let $C=\varinjlim_{\xi\in\Xi}T_\xi$ be a representation of
the $R$\+module $C$ as a colimit of countably presentable $R$\+modules
$T_\xi$ indexed by an $\aleph_1$\+directed poset~$\Xi$.
 Notice that we do \emph{not} assume any comultiplication maps~$\nu$
on the $R$\+modules~$T_\xi$.
 However, we know that $C\ot_RC=\varinjlim_{\xi\in\Xi}T_\xi\ot_RT_\xi$.

 Indeed, the tensor product functor~$\ot_R$ preserves colimits
in both of its arguments (e.~g., since it has a right adjoint functor
$\Hom_R({-},{-})$).
 So we have $C\ot_RC=\varinjlim_{\alpha\in\Xi}T_\alpha\ot_RC=
\varinjlim_{\beta\in\Xi}\varinjlim_{\alpha\in\Xi}
T_\alpha\ot_RT_\beta=\varinjlim_{(\alpha,\beta)\in\Xi^2}
T_\alpha\ot_RT_\beta$.
 Here the partial order on the Cartesian square $\Xi^2=\Xi\times\Xi$
is defined by the rule $(\alpha,\beta)\le(\alpha',\beta')$ if
$\alpha\le\alpha'$ and $\beta\le\beta'$.
 It remains to point out that the poset $\Xi$, embedded diagonally
into $\Xi^2$, is a cofinal subposet in $\Xi^2$, since the poset $\Xi^2$
is directed and for any pair of elements $(\alpha,\beta)\in\Xi^2$ there
exists an element $\xi\in\Xi$ such that $(\alpha,\beta)\le(\xi,\xi)$
in $\Xi^2$; cf.~\cite[Section~0.11 and Exercise~1.o(3)]{AR}.

 Since the $R$\+module $S$ is countably presentable and the poset $\Xi$
is $\aleph_1$\+directed, the $R$\+module map $S\rarrow C$ factorizes
as $S\rarrow T_{\xi_0}\rarrow C$ for some index $\xi_0\in\Xi$.
 Now consider the composition $T_{\xi_0}\rarrow C\rarrow C\ot_RC$.
 Since the $R$\+module $T_{\xi_0}$ is countably presentable, the poset
$\Xi$ is $\aleph_1$\+directed, and $C\ot_RC=\varinjlim_{\xi\in\Xi}
T_\xi\ot_RT_\xi$, we can find an index $\xi_1\ge\xi_0$ in $\Xi$ such
that the composition $T_{\xi_0}\rarrow C\rarrow C\ot_RC$ factorizes
through the $R$\+module morphism $T_{\xi_1}\ot_RT_{\xi_1}\rarrow
C\ot_RC$.
 So we obtain an $R$\+module map $T_{\xi_0}\rarrow T_{\xi_1}\ot_R
T_{\xi_1}$.

 Next we consider the composition $T_{\xi_1}\rarrow C\rarrow C\ot_RC$.
 As in the previous paragraph, we can find an index $\xi'_2\ge\xi_1$
in $\Xi$ such that the composition $T_{\xi_1}\rarrow C\rarrow C\ot_RC$
factorizes through the $R$\+module map $T_{\xi'_2}\ot_R T_{\xi'_2}
\rarrow C\ot_RC$.
 So we obtain an $R$\+module map $T_{\xi_1}\rarrow T_{\xi'_2}\ot_R
T_{\xi'_2}$.
 Now, at this point, the leftmost square of the diagram
$$
 \xymatrix{
  {T_{\xi_0}} \ar[r] \ar[d] & {T_{\xi_1}} \ar[r] \ar[d] & C \ar[d] \\
  {T_{\xi_1}\ot_RT_{\xi_1}} \ar[r] & {T_{\xi'_2}\ot_R T_{\xi'_2}}
  \ar[r] & {C\ot_RC}
 }
$$
\emph{need not} be commutative, while the rightmost and the outer
squares are commutative.
 Since the $R$\+module $T_{\xi_0}$ is countably presentable, the poset
$\Xi$ is $\aleph_1$\+directed, and $C\ot_RC=\varinjlim_{\xi\in\Xi}
T_\xi\ot_RT_\xi$, we can find an index $\xi_2\ge\xi'_2$ in $\Xi$ such
that the whole diagram
$$
 \xymatrix{
  {T_{\xi_0}} \ar[r] \ar[d] & {T_{\xi_1}} \ar[r] \ar[d] & C \ar[d] \\
  {T_{\xi_1}\ot_RT_{\xi_1}} \ar[r] & {T_{\xi_2}\ot_R T_{\xi_2}}
  \ar[r] & {C\ot_RC}
 }
$$
becomes commutative.

 Proceeding in this way, we construct an $\aleph_0$\+indexed chain
of indices $\xi_0\le\xi_1\le\xi_2\le\dotsb$ in $\Xi$ and a compatible
sequence of $R$\+module maps $T_{\xi_i}\rarrow T_{\xi_{i+1}}
\ot_R T_{\xi_{i+1}}$.
 It remains to put $U=\varinjlim_{i\in\aleph_0}T_{\xi_i}$.

\smallskip
 (2)~Now, in order to illustrate the workings of the proof of
Theorem~1.3 (as spelled out in~\cite[proof of Theorem~3.1]{Pacc}),
suppose that we have managed to prove that every noncoassociative
$R$\+coalgebra $(D,\nu_D)$ is an $\aleph_1$\+directed colimit of
noncoassociative $R$\+coalgebras $(T,\nu_T)$ with countably presentable
underlying $R$\+modules~$T$.
 It is easy to see that any such $R$\+coalgebra $(T,\nu_T)$ is
an $\aleph_1$\+presentable object of the category of noncoassociative
$R$\+coalgebras~\cite[Proposition~4.2]{Pacc}.
 Let $(C,\mu_C)$ be a coassociative $R$\+coalgebra.
 How does one approach $C$ with coassociative $R$\+coalgebras
$(U,\mu_U)$ with countably presentable underlying $R$\+modules~$U$\,?

 Suppose given a coassociative $R$\+coalgebra $(C,\mu_C)$,
a noncoassociative $R$\+coalgebra $(S,\nu_S)$ with a countably
presentable underlying $R$\+module $S$, and an $R$\+coalgebra
morphism $S\rarrow C$.
 Let us explain, following~\cite[proof of Proposition~3.2]{Pacc},
how to construct a factorization of the morphism $S\rarrow C$ through
a coassociative $R$\+coalgebra $(U,\mu_U)$ with a countably presentable
underlying $R$\+module~$U$.

 Let $C=\varinjlim_{\xi\in\Xi}T_\xi$ be a representation of
the $R$\+coalgebra $C$ as a colimit of noncoassociative $R$\+coalgebras
$(T_\xi,\nu_{T_\xi})$, with countably presentable underlying
$R$\+modules $T_\xi$, indexed by an $\aleph_1$\+directed poset~$\Xi$.
 Then the $R$\+coalgebra map $(S,\nu_S)\rarrow(C,\mu_C)$ factorizes
through the $R$\+coalgebra map $(T_{\xi_0},\nu_{T_{\xi_0}})\rarrow
(C,\mu_C)$ for some $\xi_0\in\Xi$.

 Now the two compositions
$$
 \xymatrix{
  {T_{\xi_0}} \ar[r]
  & {T_{\xi_0}\ot_R T_{\xi_0}} \ar@<2pt>[r] \ar@<-2pt>[r]
  & {T_{\xi_0}\ot_R T_{\xi_0}\ot_R T_{\xi_0}}
 }
$$
\emph{need not} be equal to each other, as the coalgebra $T_{\xi_0}$
is not coassociative.
 However, the two compositions
$$
 \xymatrix{
  {T_{\xi_0}} \ar[r]
  & {T_{\xi_0}\ot_R T_{\xi_0}} \ar@<2pt>[r] \ar@<-2pt>[r]
  & {T_{\xi_0}\ot_R T_{\xi_0}\ot_R T_{\xi_0}} \ar[r]
  & {C\ot_R C\ot_RC}
 }
$$
\emph{are} equal to each other, as the coalgebra $C$ is coassociative
and $T_{\xi_0}\rarrow C$ is an $R$\+coalgebra morphism.
 Since the $R$\+module $T_{\xi_0}$ is countably presentable, the poset
$\Xi$ is $\aleph_1$\+directed, and $C\ot_RC\ot_RC=\varinjlim_{\xi\in\Xi}
T_\xi\ot_RT_\xi\ot_RT_\xi$, we can find an index $\xi_1\ge\xi_0$ in
$\Xi$ such that the two compositions
$$
 \xymatrix{
  {T_{\xi_0}} \ar[r]
  & {T_{\xi_0}\ot_R T_{\xi_0}} \ar@<2pt>[r] \ar@<-2pt>[r]
  & {T_{\xi_0}\ot_R T_{\xi_0}\ot_R T_{\xi_0}} \ar[r]
  & {T_{\xi_1}\ot_R T_{\xi_1}\ot_R T_{\xi_1}}
 }
$$
are equal to each other.

 Proceeding in this way, we construct an $\aleph_0$\+indexed chain
of indices $\xi_0\le\xi_1\le\xi_2\le\dotsb$ in $\Xi$ such that,
for every $i\ge0$, the two compositions
$$
 \xymatrix{
  {T_{\xi_i}} \ar[r]
  & {T_{\xi_i}\ot_R T_{\xi_i}} \ar@<2pt>[r] \ar@<-2pt>[r]
  & {T_{\xi_i}\ot_R T_{\xi_i}\ot_R T_{\xi_i}} \ar[r]
  & {T_{\xi_{i+1}}\ot_R T_{\xi_{i+1}}\ot_R T_{\xi_{i+1}}}
 }
$$
are equal to each other.
 It remains to put $U=\varinjlim_{i\in\aleph_0}T_{\xi_i}$ and
$\mu_U=\varinjlim_{i\in\aleph_0}\nu_{T_{\xi_i}}$.
\end{rem}

\Section{Flat Coalgebras over a Commutative Ring}

 All the results of the previous Section~\ref{coalgebras-secn}, with
the exception of Lemma~\ref{coalgebras-cocomplete}, remain valid for
the full subcategories in the respective coalgebra categories
consisting of the coalgebras that are \emph{flat as $R$\+modules}.
 The only difference is that the local $\aleph_1$\+presentability
claims need to be replaced by the $\aleph_1$\+accessibility.
 Let us spell out some details.
 
 We denote by $R\Coalg_\fl\subset R\Coalg$ and $R\Cocom_\fl\subset
R\Cocom$ the full subcategories of coalgebras that are flat as
$R$\+modules.
 So $R\Coalg_\fl$ is the category of $R$\+flat coassociative,
counital $R$\+coalgebras and $R\Cocom_\fl\subset R\Coalg_\fl$ is
the full subcategory of $R$\+flat cocommutative $R$\+coalgebras.

\begin{thm} \label{flat-coalgebras-accessible-theorem}
 Let $R$ be a commutative ring. \par
\textup{(a)} The category of coassociative, counital $R$\+flat
$R$\+coalgebras $R\Coalg_\fl$ is\/ $\aleph_1$\+ac\-ces\-si\-ble.
 The\/ $\aleph_1$\+presentable objects of $R\Coalg_\fl$ are precisely
all the $R$\+flat $R$\+co\-al\-ge\-bras whose underlying $R$\+modules
are countably presentable as objects of $R\Modl$. \par
\textup{(b)} The category of coassociative, cocommutative, counital
$R$\+flat $R$\+coalgebras $R\Cocom_\fl$ is\/ $\aleph_1$\+accessible.
 The\/ $\aleph_1$\+presentable objects of $R\Cocom_\fl$ are precisely
all the $R$\+flat cocommutative $R$\+coalgebras whose underlying
$R$\+modules are countably presentable as objects of $R\Modl$.
\end{thm}

\begin{proof}
 The argument is similar to the proof of
Theorem~\ref{coalgebras-locally-presentable-theorem}, with suitable
changes.
 Let us spell out the proof of part~(b).

 First we apply Theorem~\ref{inserter-theorem} (for $\kappa=\aleph_1$
and $\lambda=\aleph_0$).
 Put $\sK=R\Modl_\fl$ and $\sL=R\Modl\times R\Modl$.
 Let $F\:\sK\rarrow\sL$ be the functor taking a flat $R$\+module $C$ to
the pair of $R$\+modules $(C,C)\in R\Modl\times R\Modl$, and let
$G\:\sK\rarrow\sL$ be the functor taking the $R$\+module $C$ to
the pair of $R$\+modules $(C\ot_RC,\>R)\in R\Modl\times R\Modl$.
 Then the inserter category $\sD$ is the category of flat $R$\+modules
$C$ endowed with two $R$\+linear maps $\mu\:C\rarrow C\ot_RC$ and
$\epsilon\:C\rarrow R$.

 By Propositions~\ref{modules-locally-presentable}\+-%
\ref{flat-modules-accessible} and~\ref{product-proposition}, both
the categories $\sK$ and $\sL$ are $\aleph_1$\+accessible.
 The assumptions of Theorem~\ref{inserter-theorem} are satisfied, so
we can conclude that the category $\sD$ is $\aleph_1$\+accessible.
 We also obtain a description of the full subcategory of
$\aleph_1$\+presentable objects in~$\sD$.

 Now we apply Theorem~\ref{equifier-theorem} (for $\kappa=\aleph_1$
and $\lambda=\aleph_0$).
 Put $\sK=\sD$ and $\sL=(R\Modl)^4=R\Modl\times R\Modl\times R\Modl
\times R\Modl$.
 Let $F\:\sK\rarrow\sL$ be the functor taking a triple $(C,\mu,\epsilon)
\in\sD$ to the quadruple of $R$\+modules $(C,C,C,C)\in\sL$,
and let $G\:\sK\rarrow\sL$ be the functor taking the same object
$(C,\mu,\epsilon)\in\sD$ to the quadruple of $R$\+modules
$(C\ot_RC\ot_RC,\>C,\>C,\>C\ot_RC)\in\sL$.
 The natural transformations $\phi$ and~$\psi$ are the same as in
the proof of Theorem~\ref{comonoids-accessible-theorem}.

 Then the equifier category $\sE$ is the category of $R$\+flat,
coassociative, counital, cocommutative coalgebras $C$ over~$R$,
that is $\sE=R\Cocom_\fl$.
 The assumptions of Theorem~\ref{equifier-theorem} are satisfied,
and we can conclude that the category $\sE=R\Cocom_\fl$ is
$\aleph_1$\+accessible.
 Theorem~\ref{equifier-theorem} also provides the desired description
of the full subcategory of $\aleph_1$\+presentable objects in~$\sE$.

 The same argument is applicable in the case of noncocommutative
coalgebras (part~(a) of the theorem).
 One just needs to drop the elements related to cocommutativity in
the proof above.

 Alternatively, one can argue similarly to the proof of
Theorem~\ref{bimodule-flat-corings-accessible-theorem} below,
using Theorem~\ref{coalgebras-locally-presentable-theorem} as
a black box and applying Theorem~\ref{pseudopullback-theorem}.
 Another alternative approach is to apply
Theorem~\ref{comonoids-accessible-theorem} to
the monoidal category of flat $R$\+modules $\sM=R\Modl_\fl$ with
respect to the operation of tensor product $\ot=\ot_R$.
\end{proof}

\begin{rem} \label{other-flat-coalgebras-accessible-remark}
 Similarly to Remark~\ref{other-coalgebras-locally-presentable-remark},
numerous variations on the theme of
Theorem~\ref{flat-coalgebras-accessible-theorem} are possible,
producing other examples of $\aleph_1$\+accessible categories of
$R$\+flat $R$\+coalgebras.
 Let us mention some of them.

\smallskip
 (1)~Noncounital and/or noncoassociative $R$\+flat $R$\+coalgebras also
form $\aleph_1$\+ac\-ces\-si\-ble categories.
 All one needs to do in these cases is to drop the related elements
from the proof of Theorem~\ref{flat-coalgebras-accessible-theorem},
or refer to Remark~\ref{other-comonoids-accessible-remark}.

\smallskip
 (2)~The category of $R$\+flat Lie coalgebras over an arbitrary
commutative ring $R$ is $\aleph_1$\+accessible.
 This is similar to
Remark~\ref{other-coalgebras-locally-presentable-remark}(2).

\smallskip
 (3)~The category of $R$\+flat conilpotent coalgebras over any
commutative ring $R$ is $\aleph_1$\+accessible.
 This is similar to
Remark~\ref{other-coalgebras-locally-presentable-remark}(3).

\smallskip
 (4)~Concerning flat DG\+coalgebras, there are two natural versions of
such a concept.
 One can consider DG\+coalgebras $C^\bu$ over $R$ whose underlying
complex of $R$\+modules is \emph{termwise} flat (i.~e., every
$R$\+module $C^n$ if flat, for $n\in\boZ$), or one can consider
the DG\+coalgebras whose underlying complex $C^\bu$ is a \emph{homotopy
flat} complex of flat $R$\+modules.

 In the former context, one needs to know that the category of complexes
of flat $R$\+modules is $\aleph_1$\+accessible, and the complexes of
flat countably presentable $R$\+modules are precisely all
the $\aleph_1$\+presentable objects of this
category~\cite[Corollaries~10.3\+-10.4]{Pacc},
\cite[Proposition~2.4]{PS6}.

 In the latter context, one needs the use the result that the category
of homotopy flat complexes of flat $R$\+modules is
$\aleph_1$\+accessible, and the homotopy flat complexes of countably
presentable flat $R$\+modules are precisely all
the $\aleph_1$\+presentable objects of this category.
 This is based on~\cite[Theorem~1.1]{CH};
see~\cite[Proposition~2.7]{PS6}.

 With these preliminary observations in mind, the arguments are similar
to the proof of Theorem~\ref{flat-coalgebras-accessible-theorem}.
 One just needs to replace the category of flat $R$\+modules
$R\Modl_\fl$ by the category of complexes of flat $R$\+modules or
homotopy flat complexes of flat $R$\+modules, as desired (and replace
all mentions of the category of $R$\+modules $R\Modl$ by the category
of complexes of $R$\+modules).
\end{rem}

\Section{Corings over an Associative Ring}

 Let $R$ be an associative ring (which may or may not be commutative).
 A (\emph{coassociative, counital}) \emph{coring} over $R$ is
a comonoid object in the monoidal category of $R$\+$R$\+bimodules
$R\Bimod R$ with respect to the operation of tensor product over~$R$.
 In other words, a coring $C$ is an $R$\+$R$\+bimodule endowed with
two maps of \emph{comultiplication} $\mu\:C\rarrow C\ot_RC$ and
\emph{counit} $\epsilon\:C\rarrow R$, which must be $R$\+$R$\+bimodule
maps satisfying the coassociativity and counitality axioms.

 We suppress the explicit diagrams and equations, as they are written
down exactly the same as in the definition of a coassociative,
counital $R$\+coalgebra in Section~\ref{coalgebras-secn}.
 The difference between the two definitions is that $C$ was
an $R$\+module (for a commutative ring~$R$) in
Section~\ref{coalgebras-secn}, while $C$ is an $R$\+$R$\+bimodule
(for an associative ring~$R$) in the present section.

 Assume further that $R$ is an associative, unital algebra over
a fixed commutative ring~$k$.
 An \emph{$R$\+$R$\+bimodule over~$k$} is an $R$\+$R$\+bimodule in
which the left and right actions of~$k$ agree.
 If a coring $C$ over $R$ is an $R$\+$R$\+bimodule over~$k$, we will
say that $C$ is an \emph{$R$\+coring over~$k$} (then it follows
automatically that the comultiplication and counit $\mu$ and~$\epsilon$
are $k$\+linear maps).
 From now on, we will assume that all our $R$\+$R$\+bimodules are
$R$\+$R$\+bimodules over~$k$.
 The category of $R$\+$R$\+bimodules over~$k$ will be denoted by
$R_k\Bimod{}_kR=(R\ot_kR^\rop)\Modl$.

 Morphisms of $R$\+corings over~$k$ are defined in the obvious way
as $R$\+$R$\+bimodule maps compatible with the comultiplication
and counit.
 We denote the category of $R$\+corings over~$k$ by $R_k\Corings$.

\begin{lem} \label{corings-cocomplete}
 All colimits exist in the category $R_k\Corings$.
 The forgetful functor $R_k\Corings\rarrow R_k\Bimod{}_kR$ preserves
colimits.
\end{lem}

\begin{proof}
 Similar to the proof of Lemma~\ref{coalgebras-cocomplete}, with
$R$\+module maps replaced by $R$\+$R$\+bi\-mod\-ule maps.
\end{proof}

\begin{thm} \label{corings-locally-presentable-theorem}
 Let $k$~be a commutative ring and $R$ be an associative, unital
$k$\+algebra.
 Then the category $R_k\Corings$ of coassociative, counital
$R$\+corings over~$k$ is locally\/ $\aleph_1$\+presentable.
 The\/ $\aleph_1$\+presentable objects of $R_k\Corings$ are precisely
all the $R$\+corings whose underlying $R$\+$R$\+bimodules are countably
presentable as objects of $R_k\Bimod{}_kR=(R\ot_kR^\rop)\Modl$.
\end{thm}

\begin{proof}
 This is a particular case of~\cite[Example~4.10]{Ulm}.
 Our argument is similar to the proof of
Theorem~\ref{coalgebras-locally-presentable-theorem}(a).
 In view of Lemma~\ref{corings-cocomplete}, it suffices to show that
the category $R_k\Corings$ is $\aleph_1$\+accessible and describe its
full subcategory of $\aleph_1$\+presentable objects.

 Firstly we apply Theorem~\ref{inserter-theorem} (for $\kappa=\aleph_1$
and $\lambda=\aleph_0$).
 Put $\sK=R_k\Bimod{}_kR$ and $\sL=(R_k\Bimod{}_kR)\times
(R_k\Bimod{}_kR)$.
 Let $F\:\sK\rarrow\sL$ be the functor taking an $R$\+$R$\+bimodule $C$
to the pair of $R$\+$R$\+bimodules $(C,C)\in R_k\Bimod{}_kR\times
R_k\Bimod{}_kR$, and let $G\:\sK\rarrow\sL$ be the functor taking
the $R$\+$R$\+bimodule $C$ to the pair of $R$\+$R$\+bimodules
$(C\ot_RC,\>R)\in R_k\Bimod{}_kR\times R_k\Bimod{}_kR$.
 Then the inserter category $\sD$ is the category of $R$\+$R$\+bimodules
$C$ over~$k$ endowed with two $R$\+$R$\+bimodule maps
$\mu\:C\rarrow C\ot_RC$ and $\epsilon\:C\rarrow R$.

 By Proposition~\ref{modules-locally-presentable} (for the ring
$R\ot_kR^\rop$) and Proposition~\ref{product-proposition}, both
the categories $\sK$ and $\sL$ are $\aleph_1$\+accessible (in fact,
locally $\aleph_1$\+presentable).
 The assumptions of Theorem~\ref{inserter-theorem} are satisfied,
and we can conclude that the category $\sD$ is $\aleph_1$\+accessible.
 Notice that, in the assumptions of Theorem~\ref{inserter-theorem},
the functor $G$ \emph{need not} take $\kappa$\+presentable objects to
$\kappa$\+presentable objects; and in the situation at hand it doesn't.
 The functor $F$ \emph{must} take $\kappa$\+presentable objects to
$\kappa$\+presentable objects; and in the situation at hand it does.
 Theorem~\ref{inserter-theorem} also provides a description of the full
subcategory of $\aleph_1$\+presentable objects in~$\sD$.

 Secondly we apply Theorem~\ref{equifier-theorem}, continuing to argue
exactly as in the proof of
Theorem~\ref{coalgebras-locally-presentable-theorem}(a) with
the category $R\Modl$ replaced by $R_k\Bimod{}_kR$ everywhere.
 So we put $\sK=\sD$ and $\sL=(R_k\Bimod{}_kR)^3$, etc.
 Once again, for applicability of
Theorem~\ref{equifier-theorem}, the functor $F$ has to take
$\aleph_1$\+presentable objects to $\aleph_1$\+presentable objects
(and it does), while the functor $G$ need not have this property
(and doesn't).
 So the assumptions of Theorem~\ref{equifier-theorem} are satisfied,
and we conclude that the category $\sE=R_k\Corings$ is
$\aleph_1$\+accessible.
 Theorem~\ref{equifier-theorem} also provides the desired description
of the full subcategory of $\aleph_1$\+presentable objects in~$\sE$.

 Alternatively, the assertion about $\aleph_1$\+accessibility of
the category $R_k\Corings$ together with the description
of $\aleph_1$\+presentable objects in this category can be obtained
as a particular case of Theorem~\ref{comonoids-accessible-theorem}(a)
for the monoidal category $\sM=R_k\Bimod{}_kR$ with respect to
the operation of tensor product $\ot=\ot_R$.
\end{proof}

\begin{rem} \label{other-corings-locally-presentable-remark}
 Similarly to Remarks~\ref{other-comonoids-accessible-remark}
and~\ref{other-coalgebras-locally-presentable-remark},
one can think of some variations on the theme of
Theorem~\ref{corings-locally-presentable-theorem}, producing locally
$\aleph_1$\+presentable categories of corings with the simplest
expected description of the full subcategories of
$\aleph_1$\+presentable objects by virtue of essentially
the same argument.
 This applies to noncoassociative and/or noncounital $R$\+corings,
DG\+corings over~$R$, etc.
\end{rem}

\Section{Bimodule-Flat Corings} \label{bimodule-flat-secn}

 Let $k$~be a commutative ring and $R$ be an associative, unital
$k$\+algebra.
 As in the previous section, we consider $R$\+corings $C$ over~$k$,
i.~e., $R$\+$R$\+bimodules over~$k$ endowed with a coassociative,
counital coring structure.

 We will say that an $R$\+coring $C$ is \emph{$R$\+$R$\+bimodule
flat} (or ``bimodule flat'' for brevity) if $C$ is flat as a module
over the ring $R\ot_k R^\rop$.
 Let us denote the category of $R$\+$R$\+bimodule flat corings over~$k$
by $R_k\Corings_\bifl$.

\begin{thm} \label{bimodule-flat-corings-accessible-theorem}
 Let $k$~be a commutative ring and $R$ be an associative, unital
$k$\+algebra.
 Then the category $R_k\Corings_\bifl$ of $R$\+$R$\+bimodule flat
coassociative, counital $R$\+corings over~$k$ is\/
$\aleph_1$\+accessible.
 The\/ $\aleph_1$\+presentable objects of $R_k\Corings_\bifl$ are
precisely all the bimodule flat $R$\+corings whose underlying
$R$\+$R$\+bimodules are countably presentable as objects of
$R_k\Bimod{}_kR=(R\ot_kR^\rop)\Modl$.
\end{thm}

\begin{proof}
 One can argue similarly to the proofs of
Theorems~\ref{flat-coalgebras-accessible-theorem}(a) and
\ref{corings-locally-presentable-theorem}; but we prefer to
present a differently structured argument based on
Theorem~\ref{pseudopullback-theorem} and using
Theorem~\ref{corings-locally-presentable-theorem} as a black box.

 Put $\sK_1=R_k\Corings$, \ $\sK_2=(R\ot_k R^\rop)\Modl_\fl$, and
$\sL=R_k\Bimod{}_kR=(R\ot_k\nobreak R^\rop)\Modl$.
 Let $F_1\:\sK_1\rarrow\sL$ be the forgetful functor assigning to every
$R$\+coring $C$ over~$k$ its underlying $R$\+$R$\+bimodule $C$, and
let $F_2\:\sK_2\rarrow\sL$ be the identity inclusion functor.
 Then the pseudopullback category $\sC$ of the functors $F_1$ and
$F_2$ is equivalent to the desired category of bimodule flat
corings $R_k\Corings_\bifl$.

 Propositions~\ref{modules-locally-presentable}\+-%
\ref{flat-modules-accessible} and
Theorem~\ref{corings-locally-presentable-theorem} tell that
the categories $\sK_1$, $\sK_2$, and $\sL$ are $\aleph_1$\+accessible,
and describe their full subcategories of $\aleph_1$\+presentable
objects.
 Theorem~\ref{pseudopullback-theorem} (for $\kappa=\aleph_1$ and
$\lambda=\aleph_0$) is applicable to the pair of functors $F_1$
and~$F_2$.
 Theorem~\ref{pseudopullback-theorem} tells that the category $\sC$ is
$\aleph_1$\+accessible, and provides the desired description of its
full subcategory of $\aleph_1$\+presentable objects.
\end{proof}

 Notice that Theorem~\ref{bimodule-flat-corings-accessible-theorem}
is \emph{not} a particular case of
Theorem~\ref{comonoids-accessible-theorem}(a), because
$R_k\Bimod{}_kR$ is \emph{not} a monoidal category with respect to
the tensor product operation~$\ot_R$.
 The tensor product over $R$ of two flat modules over $R\ot_k R^\rop$
is not a flat module over $R\ot_k R^\rop$ in general, and
the $R$\+$R$\+bimodule~$R$ (the unit object of $R_k\Bimod{}_kR$) is
usually not a flat $R$\+$R$\+bimodule.

\begin{rem}
 Similarly to Remarks~\ref{other-comonoids-accessible-remark},
\ref{other-flat-coalgebras-accessible-remark}(1),
and~\ref{other-corings-locally-presentable-remark}, one can consider
the categories of noncoassociative and/or noncounital bimodule flat
corings.
 The obvious versions of
Theorem~\ref{bimodule-flat-corings-accessible-theorem} hold in
these contexts as well.
\end{rem}

\Section{Corings with Flat Kernel}

 As in the previous two sections, we consider $R$\+corings $C$
over~$k$, where $k$~is a commutative ring and $R$ is an associative,
unital $k$\+algebra.
 Inspired by~\cite[\S3]{BB} and~\cite[Definition~4.19 and
Theorem~4.20]{Kuel}, we say that an $R$\+coring $C$ over $k$
\emph{has flat kernel} if the following two conditions are satisfied:
\begin{enumerate}
\item the counit map $\epsilon\:C\rarrow R$ is surjective;
\item the kernel $\overline C$ of the counit map~$\epsilon$ is
a flat module over the ring $R\ot_kR^\rop$.
\end{enumerate}

 If a coring $C$ satisfies only condition~(1) but not necessarily~(2),
we will say that $C$ \emph{has surjective counit}.
 Let us denote the full subcategory of $R$\+corings with surjective
counits by $R_k\Corings_\sur\subset R_k\Corings$ and the full
subcategory of $R$\+corings with flat kernels by
$R_k\Corings_\obifl\subset R_k\Corings_\sur\subset R_k\Corings$.

 For the purposes of this section, we will need some further
preliminary material, continuing the discussion in
Section~\ref{prelim-secn}.
 Given a category $\sK$, let us denote by $\sK^\to$ the category of
morphisms in~$\sK$ (with commutative squares in $\sK$ as morphisms
in~$\sK^\to$).
 Denote by $\sK^\epi\subset\sK^\to$ the full subcategory in $\sK^\to$
whose objects are the epimorphisms in~$\sK$ (so the morphisms in
$\sK^\epi$ are the commutative squares in $\sK$ in which one pair
of morphisms are epimorphisms, while the other pair is formed by
arbitrary morphisms).

\begin{prop} \label{morphisms-of-modules}
 Let $S$ be an associative ring and $\kappa$~be a regular cardinal.
 Then the category $S\Modl^\to$ of morphisms of left $S$\+modules
is locally $\kappa$\+presentable.
 The $\kappa$\+presentable objects of $S\Modl^\to$ are precisely all
the morphisms of $\kappa$\+presentable left $S$\+modules. \qed
\end{prop}

\begin{proof}
 This can be viewed as a particular case of
Proposition~\ref{modules-locally-presentable} for the suitable
ring of uppertriangular $2\times2$\+matrices
$R=\left(\begin{smallmatrix}S & S \\ 0 & S\end{smallmatrix}\right)$.
 See~\cite[Lemma~10.6]{Pacc} for a much more general assertion.
\end{proof}

\begin{prop} \label{epimorphisms-of-modules}
 Let $S$ be an associative ring and $\kappa$~be a regular cardinal.
 Then the category $S\Modl^\epi$ of epimorphisms of left $S$\+modules
is locally $\kappa$\+presentable.
 The $\kappa$\+presentable objects of $S\Modl^\epi$ are precisely
all the epimorphisms of $S$\+modules $T\rarrow U$, where $T$ and $U$
are $\kappa$\+presentable left $S$\+modules.
\end{prop}

\begin{proof}
 This is~\cite[Lemma~10.7]{Pacc} or~\cite[Lemma~1.6]{Pres}.
\end{proof}

 Given a regular cardinal~$\kappa$ and a ring $S$, an $S$\+module is
said to be \emph{$<\kappa$\+generated} if it admits a set of
generators of the cardinality smaller than~$\kappa$.
 In particular, $<\aleph_1$\+generated modules are said to be
\emph{countably generated}.

 A ring $S$ is said to be \emph{left countably Noetherian} if every
left ideal in $S$ is (at most) countably generated, or equivalently,
every submodule of a countably generated left $S$\+module is
countably generated.
 In the context of this section, we are interested in the countable
Noetherianity property of the ring $S=R\ot_kR^\rop$.
 As the ring $S$ is isomorphic to its opposite ring, $S\simeq S^\rop$,
the left and right countable Noetherianity conditions
are equivalent in this case.
 So we will simply say that the ring $R\ot_kR^\rop$ is (assumed to be)
countably Noetherian.

\begin{thm} \label{corings-with-flat-kernel-accessible-theorem}
 Let $k$~be a commutative ring and $R$ be an associative, unital
$k$\+algebra.
 Assume that the ring $R\ot_kR^\rop$ is countably Noetherian.
 In this setting: \par
\textup{(a)} The category $R_k\Corings_\sur$ of $R$\+corings over~$k$
with surjective counits is\/ $\aleph_1$\+accessible.
 The\/ $\aleph_1$\+presentable objects of $R_k\Corings_\sur$ are
precisely all the $R$\+corings $C$ with surjective counits whose
underlying $R$\+$R$\+bimodules $C$ are countably presentable as objects
of $R_k\Bimod{}_kR=(R\ot_kR^\rop)\Modl$, or equivalently, countably
generated as modules over $R\ot_kR^\rop$. \par
\textup{(b)} The category $R_k\Corings_\obifl$ of $R$\+corings over~$k$
with flat kernels is\/ $\aleph_1$\+accessible.
 The\/ $\aleph_1$\+presentable objects of $R_k\Corings_\obifl$ are
precisely all $R$\+corings $C$ with flat kernels whose underlying
$R$\+$R$\+bimodules $C$ are countably presentable as objects of
$R_k\Bimod{}_kR$, or equivalently, countably generated as modules
over $R\ot_kR^\rop$.
\end{thm}

\begin{proof}
 Notice first of all that, over a countably left Noetherian ring $S$,
a left module is countably generated if and only if it is countably
presentable.
 In particular, in the situation at hand, the countable Noetherianity
condition guarantees that the $R\ot_kR^\rop$\+module $R$ is countably
presentable.
 This is important for part~(a).

 Part~(a): it is worth pointing out that, in view of
Lemma~\ref{corings-cocomplete}, the full subcategory $R_k\Corings_\sur$
is closed under colimits \emph{of nonempty diagrams} in $R_k\Corings$.
 So all colimits of nonempty diagrams exist in $R_k\Corings_\sur$;
however, the category $R_k\Corings_\sur$ has \emph{no} initial object.

 We use Theorem~\ref{corings-locally-presentable-theorem} as a black
box and apply Theorem~\ref{pseudopullback-theorem}.
 Put $\sK_1=R_k\Corings$, \ $\sK_2=(R\ot_kR^\rop)\Modl^\epi$, and
$\sL=(R\ot_kR^\rop)\Modl^\to$.
 Let $F_1\:\sK_1\rarrow\sL$ be the functor assigning to
an $R$\+coring $C$ its counit map $\epsilon\:C\rarrow R$, and let
$F_2\:\sK_2\rarrow\sL$ be the identity inclusion functor.
 Then the pseudopullback category $\sC$ of the functors $F_1$ and $F_2$
is equivalent to the desired category of corings with surjective counits
$R_k\Corings_\sur$.

 Propositions~\ref{morphisms-of-modules}\+-\ref{epimorphisms-of-modules}
and Theorem~\ref{corings-locally-presentable-theorem} tell that
the categories $\sK_1$, $\sK_2$, and $\sL$ are $\aleph_1$\+accessible,
and describe their full subcategories of $\aleph_1$\+presentable
objects.
 Theorem~\ref{pseudopullback-theorem} (for $\kappa=\aleph_1$ and
$\lambda=\aleph_0$) is applicable to the pair of functors $F_1$
and~$F_2$.
 Here one needs to know that $R$ is a countably presentable module
over $R\ot_kR^\rop$ in order to claim that the functor $F_1$ takes
$\aleph_1$\+presentable objects to $\aleph_1$\+presentable objects.
 It is clear from Lemma~\ref{corings-cocomplete} that the functor $F_1$
preserves directed colimits.
 Theorem~\ref{pseudopullback-theorem} tells that the category $\sC$ is
$\aleph_1$\+accessible, and provides the desired description of its
full subcategory of $\aleph_1$\+presentable objects.

 Part~(b): we apply Theorem~\ref{pseudopullback-theorem} again,
using part~(a) as a black box.
 Put $\sK_1=R_k\Corings_\sur$, \ $\sK_2=(R\ot_kR^\rop)\Modl_\fl$, and
$\sL=(R\ot_kR^\rop)\Modl$.
 Let $F_1\:\sK_1\rarrow\sL$ be the functor assigning the kernel
$\overline C$ of the counit map $\epsilon\:C\rarrow R$ to an $R$\+coring
$C$ with surjective counit.
 Let $F_2\:\sK_2\rarrow\sL$ be the identity inclusion functor.
 Then the pseudopullback category $\sC$ of the functors $F_1$ and $F_2$
is equivalent to the desired category of corings with flat kernels
$R_k\Corings_\obifl$.

 Propositions~\ref{modules-locally-presentable}\+-%
\ref{flat-modules-accessible} and part~(a) of the present theorem 
tell that the categories $\sK_1$, $\sK_2$, and $\sL$ are
$\aleph_1$\+accessible, and describe their full subcategories of
$\aleph_1$\+presentable objects.
 Once again we claim that Theorem~\ref{pseudopullback-theorem} (for
$\kappa=\aleph_1$ and $\lambda=\aleph_0$) is applicable to the pair
of functors $F_1$ and~$F_2$.
 Here we need to know that that the kernel $\overline C$ of
the counit map $\epsilon\:C\rarrow R$ is countably presentable as
a module over $R\ot_kR^\rop$ whenever so is the $R$\+$R$\+bimodule~$C$.
 This is where the countable Noetherianity condition is used again.
 It is also important that the functor $F_1$ preserves directed
colimits.
 Theorem~\ref{pseudopullback-theorem} tells that the category $\sC$ is
$\aleph_1$\+accessible, and provides the desired description of its
full subcategory of $\aleph_1$\+presentable objects.
\end{proof}

\Section{Module-Flat Corings}

 We continue to consider $R$\+corings $C$ over~$k$, where $k$ is
a commutative ring and $R$ is an associative, unital $k$\+algebra.
 We say that a coring $C$ is \emph{right $R$\+module flat} (or
``right flat'' for brevity) if $C$ is a flat right $R$\+module.
 The importance of this condition is explained by the results
of~\cite[Sections~18.6, 18.14, and~18.16]{BW},
\cite[Proposition~2.12(a)]{Prev}, and~\cite[Lemma~2.1]{Pflcc},
telling that the category of left $C$\+comodules $C\Comodl$ is
abelian \emph{and} the forgetful functor $C\Comodl\rarrow R\Modl$
is exact if and only if $C$ is a flat right $R$\+module.

 Furthermore, we say that a coring $C$ is \emph{left and right
$R$\+module flat} (or just ``left and right flat'') if $C$ is
a flat left $R$\+module \emph{and} a flat right $R$\+module.
 This condition should not be confused with the condition of flatness
of $C$ as a module over $R\ot_kR^\rop$ (which was discussed in
Section~\ref{bimodule-flat-secn}).
 When $R$ is a flat $k$\+module, one can say that any
$R$\+$R$\+bimodule flat $R$\+coring is left and right $R$\+module
flat, but \emph{not} vice versa.
 We denote the full subcategory of right $R$\+module flat $R$\+corings
over~$k$ by $R_k\Corings_\rfl\subset R_k\Corings$ and the full
subcategory of left and right $R$\+module flat $R$\+corings by
$R_k\Corings_\lrfl\subset R_k\Corings_\rfl\subset R_k\Corings$.

 We will also need terminology and notation for the related categories
of $R$\+$R$\+bimod\-ules over~$k$ with $R$\+module flatness conditions.
 So let $R_k\Bimodrfl{}_kR\subset R_k\Bimod{}_kR$ denote
the full subcategory of $R$\+$R$\+bimodules that are flat as
right $R$\+modules, and let $R_k\Bimodlrfl{}_kR\subset
R_k\Bimodrfl{}_kR\subset R_k\Bimod{}_kR$ denote the full
subcategory of $R$\+$R$\+bimodules that are flat both as left
$R$\+modules and as right $R$\+modules.
 The objects of $R_k\Bimodrfl{}_kR$ will be called \emph{right
$R$\+module flat $R$\+$R$\+bimodules} (or simply ``right flat
bimodules''), while the objects of $R_k\Bimodlrfl{}_kR$ will be
called \emph{left and right $R$\+module flat $R$\+$R$\+bimodules}
(or simply ``left and right flat bimodules'').

 Finally, let us denote by $\Modr R$ the category of right $R$\+modules,
and by $\Modrfl R\subset\Modr R$ the full subcategory of flat right
$R$\+modules.

 For any set $X$, we denote by $|X|$ the cardinality of~$X$.
 Notice that, for any ring $S$ and any regular cardinal
$\kappa>|S|$, an $S$\+module $M$ is $\kappa$\+presentable if and only
if $M$ is $<\kappa$\+generated, and if and only if $|M|<\kappa$.

\begin{prop} \label{one-sided-flat-bimodules}
 Let $k$~be a commutative ring, $R$ be an associative, unital
$k$\+algebra, and $\kappa>|R|+\aleph_0$ be a regular cardinal.
 In this setting: \par
\textup{(a)} The category $R_k\Bimodrfl{}_kR$ of right $R$\+module
flat $R$\+$R$\+bimodules over~$k$ is $\kappa$\+accessible.
 The $\kappa$\+presentable objects of $R_k\Bimodrfl{}_kR$ are
precisely all the right flat bimodules of the cardinality smaller
than~$\kappa$. \par
\textup{(b)} The category $R_k\Bimodlrfl{}_kR$ of left and right
$R$\+module flat $R$\+$R$\+bimodules over~$k$ is $\kappa$\+accessible.
 The $\kappa$\+presentable objects of $R_k\Bimodlrfl{}_kR$ are
precisely all the left and right flat bimodules of the cardinality
smaller than~$\kappa$.
\end{prop}

\begin{proof}
 Part~(a): apply Theorem~\ref{pseudopullback-theorem} to the following
pair of functors $F_1$ and~$F_2$.
 Put $\sK_1=R_k\Bimod{}_kR=(R\ot_kR^\rop)\Modl$, \
$\sK_2=\Modrfl R$, and $\sL=\Modr R$.
 Let $F_1\:\sK_1\rarrow\sL$ be the forgetful functor assigning to
an $R$\+$R$\+bimodule its underlying right $R$\+module, and let
$F_2\:\sK_2\rarrow\sL$ be the identity inclusion functor.
 Then the pseudopullback category $\sC$ of the functors $F_1$ and
$F_2$ is equivalent to the desired category of right $R$\+module flat
$R$\+$R$\+bimodules $R_k\Bimodrfl{}_kR$.

 Propositions~\ref{modules-locally-presentable}\+-%
\ref{flat-modules-accessible} tell that the categories $\sK_1$, $\sK_2$,
and $\sL$ are $\kappa$\+accessible, and describe their full
subcategories of $\kappa$\+presentable objects.
 We claim that Theorem~\ref{pseudopullback-theorem} (for the given
cardinal~$\kappa$ and $\lambda=\aleph_0$) is applicable to the pair
of functors $F_1$ and~$F_2$.
 The condition that $\kappa>|R|$ is needed here for the functor $F_1$
to take $\kappa$\+presentable objects to $\kappa$\+presentable objects.
 Theorem~\ref{pseudopullback-theorem} tells that the category $\sC$ is
$\kappa$\+accessible, and provides the desired description of its
full subcategory of $\kappa$\+presentable objects.

 Part~(b): once again, we apply Theorem~\ref{pseudopullback-theorem}.
 Put $\sK_1=R_k\Bimod{}_kR=(R\ot_k\nobreak R^\rop)\Modl$, \
$\sK_2=R\Modl_\fl\times\Modrfl R$, and $\sL=R\Modl\times\Modr R$.
 Let $F_1\:\sK_1\rarrow\sL$ be the forgetful functor assigning to
an $R$\+$R$\+bimodule $B$ the pair of its underlying left and right
$R$\+modules $(B,B)\in R\Modl\times\Modr R$, and let $F_2\:\sK_2
\rarrow\sL$ be the Cartesian product of the identity inclusion functors
$R\Modl_\fl\rarrow R\Modl$ and $\Modrfl R\rarrow\Modr R$.

 Propositions~\ref{modules-locally-presentable}\+-%
\ref{flat-modules-accessible} together with
Proposition~\ref{product-proposition} tell that the categories
$\sK_1$, $\sK_2$, and $\sL$ are $\kappa$\+accessible, and describe
their full subcategories of $\kappa$\+presentable objects.
 Once again, the condition that $\kappa>|R|$ implies that the functor
$F_1$ takes $\kappa$\+presentable objects to $\kappa$\+presentable
objects.
 So Theorem~\ref{pseudopullback-theorem} is applicable for
the given cardinal~$\kappa$ and $\lambda=\aleph_0$, providing
the desired conclusions.
\end{proof}

 The following proposition is a straightforward generalization of
Theorem~\ref{corings-locally-presentable-theorem}.

\begin{prop} \label{corings-kappa-locally-presentable}
 Let $k$~be a commutative ring, $R$ be an associative, unital
$k$\+algebra, and $\kappa$~be an uncountable regular cardinal.
 Then the category $R_k\Corings$ of coassociative, counital
$R$\+corings over~$k$ is locally $\kappa$\+presentable.
 The $\kappa$\+presentable objects of $R_k\Corings$ are precisely
all the $R$\+corings whose underlying $R$\+$R$\+bimodules are
$\kappa$\+presentable as objects of
$R_k\Bimod{}_kR=(R\ot_kR^\rop)\Modl$.
\end{prop}

\begin{proof}
 The same as the proof of
Theorem~\ref{corings-locally-presentable-theorem}, except that
one needs to apply Theorems~\ref{inserter-theorem}
and~\ref{equifier-theorem} for the given cardinal~$\kappa$
and $\lambda=\aleph_0$.
\end{proof}

\begin{thm} \label{module-flat-corings-kappa-accessible}
 Let $k$~be a commutative ring, $R$ be an associative, unital
$k$\+algebra, and $\kappa>|R|+\aleph_0$ be a regular cardinal.
 In this setting: \par
\textup{(a)} The category $R_k\Corings_\rfl$ of right $R$\+module
flat $R$\+corings over~$k$ is $\kappa$\+ac\-ces\-si\-ble.
 The $\kappa$\+presentable objects of $R_k\Corings_\rfl$ are
precisely all the right flat $R$\+corings of the cardinality smaller
than~$\kappa$. \par
\textup{(b)} The category $R_k\Corings_\lrfl$ of left and right
$R$\+module flat $R$\+corings over~$k$ is $\kappa$\+accessible.
 The $\kappa$\+presentable objects of $R_k\Corings_\lrfl$ are
precisely all the left and right flat $R$\+corings of the cardinality
smaller than~$\kappa$.
\end{thm}

\begin{proof}
 This is similar to
Theorem~\ref{bimodule-flat-corings-accessible-theorem}.
 We sketch the argument based on Theorem~\ref{pseudopullback-theorem}.

 Part~(a): put $\sK_1=R_k\Corings$, \ $\sK_2=R_k\Bimodrfl{}_kR$,
and $\sL=R_k\Bimod{}_kR=(R\ot_k\nobreak R^\rop)\Modl$.
 Let $F_1\:\sK_1\rarrow\sL$ be the forgetful functor assigning to every
$R$\+coring $C$ over~$k$ its underlying $R$\+$R$\+bimodule $C$, and
let $F_2\:\sK_2\rarrow\sL$ be the identity inclusion functor.
 Then the pseudopullback category $\sC$ of the functors $F_1$ and
$F_2$ is equivalent to the desired category of right flat corings
$R_k\Corings_\rfl$.

 Propositions~\ref{modules-locally-presentable},
\ref{one-sided-flat-bimodules}(a),
and~\ref{corings-kappa-locally-presentable} tell that
the categories $\sK_1$, $\sK_2$, and $\sL$ are $\kappa$\+accessible,
and describe their full subcategories of $\kappa$\+presentable objects.
 Theorem~\ref{pseudopullback-theorem} (for the given cardinal~$\kappa$
and $\lambda=\aleph_0$) is applicable to the pair of functors $F_1$
and~$F_2$.
 Theorem~\ref{pseudopullback-theorem} tells that the category $\sC$ is
$\kappa$\+accessible, and provides the desired description of its
full subcategory of $\kappa$\+presentable objects.

 The proof of part~(b) is similar, except that one needs to take
$\sK_2=R_k\Bimodlrfl{}_kR$ and use
Proposition~\ref{one-sided-flat-bimodules}(b).
 Alternatively, one can apply
Theorem~\ref{comonoids-accessible-theorem}(a)
to the monoidal category $\sM=R_k\Bimodrfl{}_kR$ (for part~(a)) or
$\sM=R_k\Bimodlrfl{}_kR$ (for part~(b)) with respect to the operation
of tensor product $\ot=\ot_R$.
\end{proof}

\begin{rem}
 In the context of this section, the purification-based approach in
the spirit of~\cite[Theorem~3.1]{Bar} actually gives a better result
than our Theorem~\ref{module-flat-corings-kappa-accessible},
in that it represents a coring $C$ with a flatness condition as
a $\kappa$\+directed \emph{union} of its $R$\+$R$\+bimodule pure
\emph{subcorings} $C'$ satisfying the same flatness condition and
having cardinalities smaller than~$\kappa$.
 Similarly, the purification-based approach gives a better result
than Proposition~\ref{one-sided-flat-bimodules}, in that it represents
a bimodule $B$ with a flatness condition as a $\kappa$\+directed union
of its pure \emph{subbimodules} $B'$ satisfying the same flatness
condition and having cardinalities smaller than~$\kappa$.
 Notice that, for any ring homomorphism $R\rarrow S$ and any
$S$\+module $M$, any pure $S$\+submodule of $M$ is also pure as
an $R$\+submodule of $M$; in particular, any pure submodule over
$R\ot_kR^\rop$ is a pure left and right $R$\+submodule as well.
 So this last section is included in this paper only for illustrative
and comparison purposes, as well as to point out an apparent
limitation of our methods.
\end{rem}

\end{document}